\documentstyle[10pt,amscd,xypic,amssymb,combelow]{amsart}

\xyoption{all}
\CompileMatrices

\newcommand{\nc}{\newcommand}

\makeatletter
\@addtoreset{equation}{section}
\makeatother

\newenvironment{proof}{{\noindent \textbf{Proof}\,\,}}{\hspace*{\fill}$\Box$\medskip}

\newtheorem{theorem}[equation]{Theorem}
\newtheorem{exercise}[equation]{Exercise}
\newtheorem{proposition}[equation]{Proposition}
\newtheorem{lemma}[equation]{Lemma}
\newtheorem{claim}[equation]{Claim}

\theoremstyle{definition}

\theoremstyle{remark}

\nc{\fa}{{\mathfrak{a}}}
\nc{\fb}{{\mathfrak{b}}}
\nc{\fg}{{\mathfrak{g}}}
\nc{\fh}{{\mathfrak{h}}}
\nc{\fj}{{\mathfrak{j}}}
\nc{\fn}{{\mathfrak{n}}}
\nc{\fu}{{\mathfrak{u}}}
\nc{\fp}{{\mathfrak{p}}}
\nc{\fr}{{\mathfrak{r}}}
\nc{\ft}{{\mathfrak{t}}}
\nc{\fsl}{{\mathfrak{sl}}}
\nc{\fgl}{{\mathfrak{gl}}}
\nc{\hsl}{{\widehat{\mathfrak{sl}}}}
\nc{\hgl}{{\widehat{\mathfrak{gl}}}}
\nc{\hg}{{\widehat{\mathfrak{g}}}}
\nc{\chg}{{\widehat{\mathfrak{g}}}{}^\vee}
\nc{\hn}{{\widehat{\mathfrak{n}}}}
\nc{\chn}{{\widehat{\mathfrak{n}}}{}^\vee}

\nc{\fC}{{\mathfrak{C}}}
\nc{\fZ}{{\mathfrak{Z}}}

\nc{\pol}{{\text{Poles}}}

\nc{\BA}{{\mathbb{A}}}
\nc{\BC}{{\mathbb{C}}}
\nc{\BM}{{\mathbb{M}}}
\nc{\BN}{{\mathbb{N}}}
\nc{\BQ}{{\mathbb{Q}}}
\nc{\BF}{{\mathbb{F}}}
\nc{\BK}{{\mathbb{K}}}
\nc{\BP}{{\mathbb{P}}}
\nc{\BR}{{\mathbb{R}}}
\nc{\BZ}{{\mathbb{Z}}}

\nc{\CA}{{\mathcal{A}}}
\nc{\CB}{{\mathcal{B}}}
\nc{\CE}{{\mathcal{E}}}
\nc{\CF}{{\mathcal{F}}}
\nc{\tCF}{{\widetilde{\CF}}}
\nc{\oCF}{{\overline{\CF}}}
\nc{\CG}{{\mathcal{G}}}
\nc{\CI}{{\mathcal{I}}}
\nc{\CL}{{\mathcal{L}}}
\nc{\CM}{{\mathcal{M}}}
\nc{\CH}{{\mathcal{H}}}
\nc{\CN}{{\mathcal{N}}}
\nc{\CO}{{\mathcal{O}}}
\nc{\CP}{{\mathcal{P}}}
\nc{\CQ}{{\mathcal{Q}}}
\nc{\CR}{{\mathcal{R}}}
\nc{\CS}{{\mathcal{S}}}
\nc{\CT}{{\mathcal{T}}}
\nc{\CU}{{\mathcal{U}}}
\nc{\CV}{{\mathcal{V}}}
\nc{\CW}{{\mathcal{W}}}
\nc{\tCW}{{\widetilde{\CW}}}
\nc{\oCW}{{\overline{\CW}}}

\nc{\uu}{{U_q(\hgl_n)}}
\nc{\uup}{{U_q^+(\hgl_n)}}
\nc{\uug}{{U_q^\geq(\hgl_n)}}
\nc{\uus}{{U_q(\hsl_n)}}
\nc{\uuu}{{U_q(\hgl_1)}}
\nc{\nn}{{\mathbb{N}}}

\nc{\tT}{{T}}

\nc{\wfZ}{{\widetilde{\fZ}}}

\nc{\od}{{\overline{d}}}
\nc{\rg}{{\textrm{R}\Gamma}}
\nc{\erg}{{\emph{R}\Gamma}}
\nc{\id}{{\textrm{id}}}

\def\ph{\varphi}

\def\kk{{{\mathbb{K}}}}

\def\ph{\varphi}

\def\high{\text{ht }}

\def\sym{\text{Sym}}

\def\degu{\deg_{\searrow} \ }
\def\degd{\deg_{\nwarrow} \ }
\def\loccit{\emph{loc cit}}

\def\od{\overline{d}}
\def\oe{e}

\def\lamu{{\lambda \backslash \mu}}

\def\ld{\text{l.d. }}
\def\hd{\text{h.d. }}
\def\out{{\text{out}}}

\def\uu{U_v(\widehat{\fgl}_1)}

\def\tG{\widetilde{G}}

\begin{document}

\title[The $\lowercase{\frac mn}$ Pieri rule]{\Large{\textbf{The $\lowercase{\frac mn}$ Pieri rule}}}

\author[Andrei Negu\cb t]{Andrei Negu\cb t}
\address{Columbia University, Department of Mathematics, New York, USA}
\address{Simion Stoilow Institute of Mathematics, Bucharest, Romania}
\email{andrei.negut@@gmail.com}

\maketitle
\thispagestyle{empty}

\begin{abstract}

The Pieri rule is an important theorem which explains how the operators $\oe_k$ of multiplication by elementary symmetric functions act in the basis of Schur functions $s_\lambda$. In this paper, for any $m/n\in \BQ$ we study the relationship between the ``rational" version of the operators: 
$$
\oe_k^{m/n} : \Lambda \longrightarrow \Lambda
$$ 
given by the elliptic Hall algebra, and the ``rational" version $s_\lambda^{m/n}$ of the basis given by the Maulik-Okounkov stable basis construction. The answer is inspired by geometry, but relevant to combinatorics and representation theory.

\end{abstract}

\section{Introduction} \noindent In this paper, we study a certain algebra $\CA$ known by many names in the theory of quantum groups: the elliptic Hall algebra, the shuffle algebra, the doubly deformed $\CW_{1+\infty}$ algebra, the Ding-Iohara algebra, the spherical type $A_\infty$ double affine Hecke algebra, and quantum toroidal $\fgl_1$. For each rational number $\frac mn \in \BQ \cup \{\infty\}$, the algebra $\CA$ contains a commutative subalgebra isomorphic to a polynomial ring in countably many variables: 
$$
\BQ(q,t)[\oe_1^{m/n}, \oe_2^{m/n},...] \cong \CA_{m/n} \subset \CA
$$
where $q,t$ are parameters. The algebra $\CA$ acts on the ring of symmetric functions:
$$
\CA \curvearrowright \Lambda = \BQ(q,t)[x_1,x_2,...]^{\sym} 
$$ 
where $e_k^0 \in \CA_0$ acts on $\Lambda$ by the operator of multiplication with the $k-$th elementary symmetric function, and $e_k^\infty \in \CA_\infty$ acts on $\Lambda$ by the $k-$th Macdonald $q-$difference operator. The full algebra $\CA$ can be thought to interpolate between these two extremes, and we will spell out the interactions between generators for different $\frac mn$. 














\textbf{} \\
The same principle applies for bases of the representation $\Lambda$. At $\frac mn = \frac 01 = 0$, we consider the basis of Schur functions, in which the operators of multiplication by symmetric functions are described quite nicely by the Pieri rule. At $\frac mn = \frac 10 = \infty$, we consider the basis of Macdonald polynomials in which the $q-$difference operators are diagonal, and hence quite presentable. We will seek to interpolate between these two bases, i.e. to define a basis:
\begin{equation}
\label{eqn:stab}
\{s_\lambda^{m/n}\}_{\lambda \text{ partition}} 
\end{equation}
of $\Lambda$, for any $\frac mn \in \BQ$. There are a number of properties one wants from such a basis, but the one we will mostly be concerned with is that the generators: 
$$
\oe_k^{m/n} \in \CA_{m/n} \subset \CA
$$ 
act ``nicely" in it. Such a choice of \eqref{eqn:stab} is given by the Maulik-Okounkov stable basis (\cite{MO2}), which we recall in Section \ref{sec:stab}. We will prove the following result: \\

\begin{theorem}
\label{thm:stab}

\textbf{(The $m/n$ Pieri rule):} For any coprime $(m,n) \in \BZ \times \BN$ and any positive integer $k$, we have:
\begin{equation}
\label{eqn:pieri}
\oe_k^{m/n} \cdot s^{m/n}_\mu = \sum s^{m/n}_\lambda  (-1)^{\emph{ht}} \prod_{i=1}^k\prod_{j=1}^n \chi_{j}(B_i)^{\left \lfloor \frac {mj}n \right \rfloor - \left \lfloor \frac {m(j-1)}n \right \rfloor} 
\end{equation}
where the sum goes over all vertical $k-$strips of $n-$ribbons of shape $\lamu$. We write $\emph{ht}$ for the height \footnote{This notion is called spin in \cite{LLT}} of such a $k-$strip, and:
$$
\chi_{j}(B) = q^{x}t^{-y}
$$
where $(x,y)$ are the coordinates of the $j-$th box in the ribbon $B$, counted in northwest-southeast direction. These notions will be defined in Subsection \ref{sub:ribbon}. \footnote{By replacing the word ``vertical" with the word ``horizontal", we obtain formulas for operators of multiplication by complete symmetric functions}

\end{theorem}

\textbf{} \\
Though our proof of Theorem \ref{thm:stab} only covers the case when $\gcd(m,n)=1$, note that when $m=0$, the operators involved are:
$$
e_k^{0/n} = \text{multiplication by }e_k(x_1^n,x_2^n,...)
$$
Since the basis $s_\lambda^0$ consists of the usual Schur functions, the analogue of \eqref{eqn:pieri} at $m=0$ coincides with the Lascoux-Leclerc-Thibon ribbon tableau formula (12) of \cite{LLT}. If we further also specialize to $n=1$, we obtain the usual Pieri rule for multiplication by elementary symmetric functions. Therefore, our result can be thought of as a deformation of the results of \cite{LLT} to general $m$, as will be explained in Section \ref{sec:action} in connection to LLT polynomials.



\textbf{} \\
Another use of Theorem \ref{thm:stab} is the particular case $k=1$ and $\mu = \emptyset$, when formula \eqref{eqn:pieri} becomes:
$$
e_1^{m/n} \cdot 1 = \sum_{i=1}^{n} s_{(i,1^{n-i})}^{m/n} \cdot q^{\sum_{j=1}^{i-1} \left \lceil \frac {mj}n \right \rceil} (-t)^{\sum_{j=1}^{n-i} \left \lfloor \frac {mj}n \right \rfloor} 
$$
The above equality gives a new interpretation of the ``symmetric function" side of the rational shuffle conjecture, and it is expected to equal the ``combinatorial side" provided by the Hikita polynomial. See \cite{GN} for a review of the rational shuffle conjecture, and also for connections with knot theory and representation theory. According to a general framework in representation theory, the above equality reflects a certain resolution of the unique finite-dimensional irreducible module of the rational Cherednik algebra (with quantization parameter $c = \frac mn$) by standard modules corresponding to hook diagrams. This should be a bigraded version of the BGG-Koszul resolution of \cite{EGL}. In general, one expects formulas \eqref{eqn:pieri} to govern the parabolic induction/restriction of standard modules for rational Cherednik algebras, endowed with an additional filtration that has not been completely defined yet (see \cite{shan} for an overview).

\textbf{} \\
The structure of this paper is the following. In Section \ref{sec:def} we recall certain basic definitions concerning symmetric functions, partitions and Young diagrams. In Section \ref{sec:ell} we recall the definition of the algebra $\CA$, its alternative presentation as a shuffle algebra, and the way the shuffle algebra helps us understand the action of $\CA$ on the ring of symmetric functions. In Section \ref{sec:stab}, we recall the definition of the stable basis $s_\lambda^{m/n}$ from \cite{MO2} and prove Theorem \ref{thm:stab}. In Section \ref{sec:action}, we show how to use formulas \eqref{eqn:pieri} to obtain LLT polynomials. 

\textbf{} \\
I would like to thank Andrei Okounkov and Davesh Maulik for teaching me a lot of beautiful mathematics, in particular their stable basis construction that constitutes the core of this paper. I would like to thank Adriano Garsia for teaching me much about the combinatorics of symmetric functions. I also thank the Research Institute for Mathematical Sciences (Kyoto) and the Japan Society for the Promotion of Science for supporting me while this paper was being written. \\

\section{Definitions and notations: symmetric functions and partitions}
\label{sec:def}

\subsection{} Much of the present paper is concerned with the ring of symmetric functions in infinitely many variables $x_1,x_2,...$, over the field $\BK = \BQ(q,t)$:
$$
\Lambda = \BK[x_1,x_2,...]^{\sym}
$$
There are a number of bases of this vector space, perhaps the most basic one consisting of monomial symmetric functions:
$$
m_\lambda = \sym \left[ x_1^{\lambda_1}x_2^{\lambda_2}... \right]
$$
where $\lambda$ goes over all partitions. Particular instances of these are the power-sum functions:
$$
p_k = m_{(k)} = x_1^k + x_2^k + ...
$$
as well as the elementary and complete symmetric functions:
$$
e_k = m_{(1,1,...,1)} = \sum_{i_1<...<i_k} x_{i_1}... x_{i_k}, \qquad \qquad h_k = \sum_{\lambda \vdash k} m_\lambda = \sum_{i_1 \leq ... \leq i_k} x_{i_1}... x_{i_k}
$$
The ring $\Lambda$ is generated by each of these particular symmetric functions:
$$
\Lambda = \BK[p_1,p_2,...] = \BK[e_1,e_2,...] = \BK[h_1,h_2,...]
$$
as an algebra, while as a vector space it is spanned by:
$$
p_\lambda = p_{\lambda_1}p_{\lambda_2}..., \qquad \text{or} \qquad e_\lambda = e_{\lambda_1}e_{\lambda_2}..., \qquad \text{or} \qquad h_\lambda = h_{\lambda_1}h_{\lambda_2}..., 
$$
as $\lambda = (\lambda_1 \geq \lambda_2 \geq ...)$ goes over all partitions of natural numbers. \\

\subsection{} We will consider two inner products on $\Lambda$, and note that both are graded, symmetric and respect the bialgebra (product and coproduct) structure of $\Lambda$. We will not go into what this means, but observe that such an inner product is uniquely determined by the pairing of $p_k$ with itself:
$$
\langle p_k,p_k \rangle = k
$$
By Gram-Schimdt, there is a unique orthogonal basis $\{s_\lambda\}$ of $\Lambda$ such that:
$$
\langle s_\lambda, s_\mu \rangle = 0 
$$
if $\lambda \neq \mu$, and:
$$
s_\lambda = m_\lambda + \sum_{\mu \lhd \lambda} m_\mu d^\mu_\lambda, \qquad d_\lambda^\mu \in \BZ
$$
where the \textbf{dominance ordering} on partitions is:
\begin{equation}
\label{eqn:dom}
\mu \unlhd \lambda \qquad \text{if} \qquad \mu_1+...+\mu_i \leq \lambda_1+...+\lambda_i \quad \forall i
\end{equation}
and $|\mu|=|\lambda|$. The symmetric polynomials $s_\lambda$ are called \textbf{Schur functions}, and they play a very important role in representation theory as the characters of irreducible representations of the special linear group. Note that we have:
\begin{equation}
\label{eqn:andra}
e_k = s_{(1,...,1)}, \qquad \quad h_k = s_{(k)}, \quad \qquad p_k = \sum_{i=0}^{k-1} (-1)^{i} s_{(k-i,1^i)}
\end{equation}

\subsection{} There is a one-to-one correspondence between partitions and Young diagrams, the latter being simply stacks of $1\times 1$ boxes placed in the corner of the first quadrant. For example, the following Young diagram:

\begin{picture}(100,160)(-90,-15)
\label{fig}

\put(17,17){$1$}
\put(15,57){$t^{-1}$}
\put(17,97){$t^{-2}$}
\put(57,17){$q$}
\put(53,57){$qt^{-1}$}
\put(97,17){$q^2$}
\put(92,57){$q^2t^{-1}$}
\put(137,17){$q^3$}

\put(0,0){\line(1,0){160}}
\put(0,40){\line(1,0){160}}
\put(0,80){\line(1,0){120}}
\put(0,120){\line(1,0){40}}

\put(0,0){\line(0,1){120}}
\put(40,0){\line(0,1){120}}
\put(80,0){\line(0,1){80}}
\put(120,0){\line(0,1){80}}
\put(160,0){\line(0,1){40}}

\put(65,-20){\mbox{Figure 1}}

\end{picture}

\text{}\\
represents the partition $(4,3,1)$, because it has 4 boxes on the first row, 3 boxes on the second row, and 1 box on the third row. The monomials displayed in Figure 1 are called the \textbf{weights} of the boxes they are in, and are defined by the formula:
\begin{equation}
\label{eqn:weight}
\chi_\square = q^{x} t^{-y}
\end{equation}
where $(x,y)$ are the coordinates of the southwest corner of the box in question. We call the integer:
\begin{equation}
\label{eqn:content}
o_\square = x-y
\end{equation}
the \textbf{content} of the box, and note that the content is constant across diagonals (in this paper, the word ``diagonal" will only refer to those in southwest-northeast direction). Finally, every box in a Young diagram comes with numbers denoted by:
$$
a(\square), \ l(\square)
$$
known as the \textbf{arm} and \textbf{leg} lengths, respectively. These numbers count the distance between the given box and the right and top borders of the partition, respectively. For example, the box of weight $t^{-1}$ in Figure 1 has arm length equal to 2 and leg length equal to 1. Moreover, a Young diagram has \textbf{inner and outer corners}: the example in Figure 1 has 4 inner corners (of weights $t^{-3},qt^{-2},q^3t^{-1},q^4$) and 3 outer corners (of weights $qt^{-3},q^3t^{-2},q^4t^{-1}$). \\

\subsection{} In this paper, we will also consider another inner product, given by:
\begin{equation}
\label{eqn:mac}
\langle p_k, p_k \rangle_{q,t} = k \cdot \frac {1 - q^k}{1 - t^k}
\end{equation}
This is known as Macdonald inner product. By the same Gram-Schimdt principle, there is a unique orthogonal basis $\{P_\lambda\}$ of $\Lambda$ such that:
$$
\langle P_\lambda, P_\mu \rangle_{q,t} = 0 
$$
if $\lambda \neq \mu$, and:
$$
P_\lambda = m_\lambda + \sum_{\mu \lhd \lambda} m_\mu c^\mu_\lambda
$$
The symmetric functions $P_\lambda$ are called \textbf{Macdonald polynomials}. We will also consider a certain renormalization of these polynomials:
\begin{equation}
\label{eqn:renormalization}
M_\lambda = \frac {P_\lambda}{\prod_{\square \in \lambda} \left(t^{-l(\square)} - q^{a(\square)+1} \right)}
\end{equation}
Let us also recall the operator $\nabla:\Lambda \longrightarrow \Lambda$ of Bergeron-Garsia which is diagonal in the basis of Macdonald polynomials:
\begin{equation}
\label{eqn:nabla}
\nabla\cdot M_\lambda = M_\lambda \prod_{\square \in \lambda} \chi_\square
\end{equation}
This operator is usually defined to be diagonal in modified Macdonald polynomials, so there will be an implicit plethysm $X\longrightarrow X (1-t)^{-1}$ connecting our notations with the more common ones in the literature. See \cite{GHT} for an overview. \\

\subsection{} 
\label{sub:ribbon}

Given two partitions, we will write $\mu \leq \lambda$ if the Young diagram of $\mu$ is completely contained in that of $\lambda$. This is equivalent with requiring that $\mu_i \leq \lambda_i$ for all $i$, and it is different from the dominance ordering \eqref{eqn:dom}. If we are in this situation, we call $\lamu$ a \textbf{skew diagram}, meaning a subset of boxes in the first quadrant obtained by removing a Young diagram from a larger one. If $\lambda\backslash \mu$ is connected and contains no $2\times 2$ square, then we call it a \textbf{ribbon}. The quantity:
$$
\high(B) = \max_{\square,\blacksquare \in B}  l(\square)-l(\blacksquare) 
$$
is called the \textbf{height} of a ribbon $B$. We will use the term $n-$ribbon if $|\lamu| = n$. The boxes of an $n-$ribbon are indexed $\square_1,...,\square_n$ going from northwest to southeast, and note that their contents are consecutive integers. Given any two disjoint $n-$ribbons, we say that one is next to the other if their first common edge is vertical, as in the following picture:

\begin{picture}(100,95)(-10,5)
\label{fig5}

\put(85,80){\line(0,1){20}}
\put(85,100){\line(1,0){40}}
\put(125,100){\line(0,-1){60}}
\put(125,40){\line(1,0){40}}
\put(165,40){\line(0,-1){20}}
\put(165,20){\line(-1,0){60}}
\put(105,20){\line(0,1){60}}
\put(105,80){\line(-1,0){20}}

\put(125,60){\line(1,0){60}}
\put(185,60){\line(0,-1){40}}
\put(185,20){\line(1,0){40}}
\put(225,20){\line(0,-1){20}}
\put(225,00){\line(-1,0){60}}
\put(165,0){\line(0,1){20}}

\linethickness{1mm}
\put(125,60){\line(0,-1){20}}

\put(250,50){\text{Figure 2}}

\end{picture}

\text{} \\
A \textbf{vertical} $k-$strip of $n-$ribbons $\{B_1,...,B_k\}$ is a collection of disjoint $n-$ribbons such that no two are next to each other. The height of such a $k-$strip is:
\begin{equation}
\label{eqn:heightstrip}
\high := \sum_{i=1}^k \high B_i
\end{equation}


\begin{lemma}
\label{lem:strip}

Any skew Young diagram $\lamu$ of size $kn$ can be covered by at most one vertical $k-$strip of $n-$ribbons. Hence the matrix coefficients of \eqref{eqn:pieri} are either zeroes or monomials $\pm q^xt^y$. \\

\end{lemma}

\begin{proof} We will prove the statement by induction on $k$, where the case $k=1$ is obvious. Let us assume a certain skew diagram $\lamu$ can be covered by a vertical $k-$strip of $n-$ribbons, and show that the covering is unique. Note that there is a unique candidate for the $n-$ribbon $B_{\out}$ which contains the northwest-most square of $\lamu$: indeed, this ribbon must start from this square and trace the external boundary of $\lamu$. We call $B_{\out}$ the \textbf{outer ribbon}, and note that if it fails to end on a right vertical boundary of $\lamu$, then we violate the condition that the skew diagram can be covered by a vertical strip of $n-$ribbons. Therefore, removing $B_{\out}$ leaves us with yet another skew Young diagram which can be covered by a $k-1$ strip of $n-$ribbons, so we can repeat the argument. Since at each step, the outer ribbon that we remove is unique, we conclude that the initial covering is unique.

\end{proof}

\subsection{} 
\label{sub:game}

Let us consider a certain \textbf{game}. Start with any skew diagram $\lamu$ and \textbf{bubble} down its rows according to the following procedure: \\

\begin{enumerate}

\item start with the topmost row (call its length $l$), and slide it diagonally in the southwest direction on top of the second row (call its length $l'$) \\

\item on the second row, we will now have two overlapping horizontal strips of boxes, of lengths:
$$
\max(l,l')+a \qquad \text{and} \qquad \min(l,l')-a \qquad \qquad \text{for some } a>0
$$

\item take the longest of the two strips and slide it diagonally on top of the next row down, and repeat the procedure \\

\item when we obtain a row of length $n$, we remove it and go back to step $(1)$

\end{enumerate}

\textbf{} \\ 
If we can remove all the boxes of $\lamu$ by applying the above sequence of moves, without ever obtaining a horizontal strip of more than $n$ boxes, we call the skew diagram $\lamu$ a \textbf{winner}. \\

\begin{lemma}
\label{lem:game}

A skew Young diagram is a winner if and only if it can be covered by a vertical strip of $n-$ribbons. \\

\end{lemma}

\begin{proof} Assume $D$ can be covered by a vertical strip of $n-$ribbons. Let $B_{\out}$ be the outer ribbon of this covering (see the proof of Lemma \ref{lem:strip}) and suppose it has height $h$. Then as we bubble down the first row according to the above procedure, after $h-1$ steps the blocks of the original ribbon $B_{\out}$ will form a horizontal strip of length $n$. We remove this strip, and then repeat the game for the remaining skew diagram, which can be covered by a vertical strip of one less $n-$ribbons.

\end{proof}

\section{The elliptic Hall and shuffle algebras}
\label{sec:ell}

\subsection{} 
\label{sub:elliptic} 

In this section, we will present a certain algebra $\CA$ known in geometric representation theory as the elliptic Hall algebra. The following is an renormalized formulation of the presentation in \cite{BS}, where this algebra was first defined. Set:
$$
\alpha_k = (1-t^{k})(1-q^{k}t^{-k}) \in \BK
$$
Let $\BN \times \BZ$ denote the right half plane lattice, and let $\BZ^\geq = \BN \times \BZ \cup \{0\times \BN\}$ denote the right half plane lattice with the positive vertical half-line included. Define the algebra $\CA^\geq$ to be generated by elements $p_v$, for all $v\in \BZ^\geq$, modulo the following relations:
\begin{equation}
\label{eqn:rel1}
[p_{kv}, p_{lv}] = 0 
\end{equation}
for any $k,l>0$ and any $v\in \BZ^\geq$, while:
\begin{equation}
\label{eqn:rel2}
[p_{v}, p_{v'}] = \theta_{v+v'}
\end{equation}
for any clockwise oriented lattice triangle $\{0, v, v + v'\} \subset \BZ^\geq$ with no lattice points inside and on the first two edges, where we define:
$$
\theta_v(z) = \sum_{k\geq 0} \theta_{kv} z^k := \frac {1 - q^{-1}}{(q - t) (1 - t^{-1})} \cdot \exp \left( \sum_{k\geq 1} \alpha_k  p_{kv} \frac {z^k}k \right)
$$
for any $v = (n,m) \in \BZ^\geq$ such that $\gcd(m,n)=1$. If the $p_v$ are thought of as power sum functions, the $\theta_v$ are plethystically modified complete symmetric functions. We will write $\CA \subset \CA^\geq$ for the subalgebra generated by $p_v$ with $v\in \BN \times \BZ \subset \BZ^\geq$. \\

\subsection{} 
\label{sub:fock}

Note that the condition \eqref{eqn:rel1} ensures that, for all coprime $(n,m) \in \BZ^\geq$:
$$
\kk[p_1^{m/n},p_2^{m/n},...] =: \CA_{m/n} \hookrightarrow \CA^\geq
$$
are all comutative subalgebras, where we denote $p_k^{m/n} = p_{kn,km}$. Our basic module for $\CA^\geq$ will be the ring of symmetric functions:
$$
\Lambda := \kk[x_1,x_2,...]^{\sym}
$$
where the algebra $\CA^\geq$ acts by:
\begin{equation}
\label{eqn:action}
p_k^0 \ = \ \text{multiplication by }p_k
\end{equation}
$$ 
p_k^r = \nabla^r p_{k}^0 \nabla^{-r}, \ \quad \forall \ r\in \BZ
$$
\begin{equation}
\label{eqn:action2}
p_k^{\infty} ( M_\lambda ) = M_\lambda \sum_{i\geq 0} q^{\lambda_i-1} t^{-i}
\end{equation}
Because of the defining relations \eqref{eqn:rel1} - \eqref{eqn:rel2}, this is enough to define the action of the whole algebra $\CA^\geq$, although one needs to check that the defining relations are met (see \cite{Nsurv} for a survey explaining the proof of this result). Hence the operators: 
$$
p_k^{m/n} = p_{kn,km}: \Lambda \longrightarrow \Lambda
$$ 
thus defined \textbf{interpolate} between the operators \eqref{eqn:action} of multiplication by $p_k$ and the Macdonald $q-$difference operators \eqref{eqn:action2}. To obtain an understanding of how these operators explicitly act on symmetric functions for general $m$ and $n$, we turn to an incarnation of $\CA$ known as the \textbf{shuffle algebra}. \\

\subsection{} Consider an infinite set of variables $z_1,z_2,...$, and take the $\kk-$vector space:
\begin{equation}
\label{eqn:big}
V = \bigoplus_{N \geq 0} \kk(z_{1},...,z_{N})^{\sym}
\end{equation}
We can endow it with a $\kk-$algebra structure by the so-called \textbf{shuffle product}:
$$
R_1(z_{1},...,z_{N}) * R_2(z_{1},...,z_{N'}) =
$$

\begin{equation}
\label{eqn:mult}
= \textrm{Sym} \left[R_1(z_{1},...,z_{N}) R_2(z_{N+1},...,z_{N+N'}) \prod_{i=1}^N \prod_{j = N+1}^{N+N'} \omega \left( \frac {z_i}{z_j} \right) \right]
\end{equation}
where:
\begin{equation}
\label{eqn:defomega}
\omega(x)  = \frac {(1 - x q)(t - x)}{(1 - x)(t - xq)}  
\end{equation}
and \textrm{Sym} denotes the symmetrization operator:
$$
\textrm{Sym}\left( R(z_1,...,z_N) \right) = \sum_{\sigma \in S(N)} R(z_{\sigma(1)},...,z_{\sigma(N)})
$$
Note that the product preserves the two gradings on the vector space $V$: the number of variables $N$ and the total homogeneous degree $M$ of rational functions. \\

\subsection{} The \textbf{shuffle algebra} is defined as the subalgebra $\CS \subset V$ consisting of rational functions of the form:
\begin{equation}
\label{eqn:shufelem}
R(z_1,...,z_N) = \frac {r(z_1,...,z_N)}{\prod_{1\leq i \neq j \leq N} (t z_i-q z_j)}
\end{equation}
where $r$ is a symmetric Laurent polynomial that satisfies the following \textbf{wheel conditions} (introduced in \cite{FHHSY}):
\begin{equation}
\label{eqn:wheel}
r(z_1,...,z_N) = 0 \qquad \text{whenever } \left\{ \frac {z_1}{z_2},\frac {z_2}{z_3},\frac {z_3}{z_1} \right\} = \left\{q,\frac 1t,\frac t{q} \right\} \text{ or } \left\{t, \frac 1{q}, \frac qt \right\} 
\end{equation}
These conditions impose quite significant restrictions on the set of elements of $\CS$, as was studied in \cite{FHHSY} and \cite{Shuf}. In particular, they ensure the following Proposition: \\

\begin{proposition}
\label{prop:well}

Consider any shuffle element $R \in \CS$ and any skew Young diagram $\lamu$. Then the following quantity is well-defined:
\begin{equation}
\label{eqn:well}
R(\lamu) := R(\chi_\square)_{\square \in \lamu} \in \BK 
\end{equation}
as long as the number of variables of $R$ equals the number of boxes of $\lamu$. \\

\end{proposition}

\subsection{} The connection between the shuffle and elliptic Hall algebras is given by: \\

\begin{theorem} (see \cite{SV}, \cite{Shuf}) \label{thm:shufhall}  We have an isomorphism of algebras:
$$
\CA \longrightarrow \CS, \qquad \qquad p_{1,m} \longrightarrow z_1^m
$$
\end{theorem}

\textbf{} \\
As a consequence of Theorem \ref{thm:shufhall}, $\CS$ is a model for the elliptic Hall algebra, whose elements are certain rational functions. One of the reasons why this is relevant is that these rational functions are precisely the kernels that describe the action of $\CA$ in the representation $\Lambda$ of Subsection \ref{sub:fock}. Explicitly, it was shown in \cite{N} that an element $R(z_1,...,z_N) \in \CS\cong \CA$ acts in the basis of modified Macdonald polynomials by the formula:
\begin{equation}
\label{eqn:shuf}
R \cdot M_\mu = \sum_{\mu \leq \lambda} M_\lambda \cdot  R(\lambda \backslash \mu) \prod_{\blacksquare \in \lambda \backslash \mu} \left[ \left(t - q \chi_\blacksquare \right) \prod_{\square \in \mu} \omega \left(\frac {\chi_{\blacksquare}}{\chi_{\square}} \right) \right] \quad \qquad
\end{equation}
where the sum goes over all skew diagrams $\lambda \backslash \mu$ of size $N$. In particular, setting $N=1$ and $R(z_1) = 1$ gives us the first $q,t-$Pieri rule for Macdonald polynomials. \\

\subsection{} 
\label{sub:p}

So to find out how the generator $p_k^{m/n}\in \CA$ acts on $\Lambda$ in the basis of modified Macdonald polynomials $M_\lambda$, we need to find out which element of the shuffle algebra it corresponds to, and then evaluate that shuffle element at the set of weights of various skew diagrams in order to use \eqref{eqn:shuf}. It was shown in \cite{Shuf} that, for any $\frac mn \in \BQ$, the isomorphism of Theorem \ref{thm:shufhall} sends:
\begin{equation}
\label{eqn:wally}
\CA \ni p^{m/n}_k \longrightarrow P_k^{m/n} \in \CS
\end{equation}
where the rational function $P_k^{m/n} (z_1,...,z_N)$ with $N=kn$ is given by $\quad P_{k}^{m/n} = $
$$
= \left[\frac {(1-t)(1-q)}{t-q}\right]^N \sym \left[\frac {\prod_{i=1}^{N} z_i^{r_{\frac mn}(i)} \sum_{i=0}^{k-1}  \frac {t^{i}z_{n}z_{2n}...z_{in}}{q^{i} z_{n+1}z_{2n+1}...z_{in+1}} }{(1-t^k)\left(1-\frac {t z_1}{q z_2} \right)...\left(1-\frac {t z_{N-1}}{q z_{N}} \right)}  \prod_{1\leq i < j \leq N} \omega\left ( \frac {z_i}{z_j} \right) \right] 
$$
where: 
$$
r_{\frac mn}(i) = \left \lceil \frac {mi}n \right \rceil - \left \lceil \frac {m(i-1)}n \right \rceil
$$
Plugging this into \eqref{eqn:shuf} gives us explicit ``shuffle formulas" for the generators of the elliptic Hall algebra $\CA$, and the way they act in the basis of Macdonald polynomials. In particular, plugging in $m=0$ and $n=1$ gives us formulas for the operators of multiplication by the power sum function $p_k$ in the basis of Macdonald polynomials. Such formulas may be considered to be $q,t-$Pieri rules for Macdonald polynomials, and the coefficients boil down to sums over standard tableaux. However, instead of working with the above explicit presentation of $P_k^{m/n}$, we will now present an implicit characterization which was developed in \cite{Shuf}. \\

\subsection{} 

For any rational number $\frac mn \in \BQ$, we have the algebra isomorphism: 
$$
\Lambda \stackrel{\cong}\longrightarrow \CA_{m/n} \subset \CA, \qquad \qquad p_k \longrightarrow p_k^{m/n}
$$
We will be interested the codomain of this isomorphism in terms of the shuffle algebra, i.e. passing through the isomorphism of Theorem \ref{thm:shufhall}:
\begin{equation}
\label{eqn:iso}
\Lambda \stackrel{\cong}\longrightarrow \CS_{m/n} \subset \CS, \qquad \qquad p_k \longrightarrow P_k^{m/n}
\end{equation}
This can be upgraded to an isomorphism of bialgebras, with respect to the usual coproduct on symmetric functions:
$$
\Delta: \Lambda \longrightarrow \Lambda \otimes \Lambda, \qquad \qquad  \Delta\left(f(x_1,x_2,...) \right) = f(x_1,x_2,... \otimes x_1',x_2',...) 
$$
In particular, we see that $\Delta(p_k) = p_k \otimes 1 + 1 \otimes p_k$, so the coproduct $\Delta$ is characterized by the property that the power sum functions are primitive. It was shown in \cite{Shuf} that the isomorphism \eqref{eqn:iso} intertwines the above coproduct with:
$$
\Delta_{m/n} : \CS_{m/n} \longrightarrow \CS_{m/n} \otimes \CS_{m/n},
$$
\begin{equation}
\label{eqn:copshuf}
\Delta_{m/n}\left(R(z_1,...,z_N) \right) = \sum_{i=0}^N \lim_{\xi \rightarrow \infty} \frac {R(\xi z_1,..., \xi z_i \otimes z_{i+1},..., z_N)}{\xi^{\frac {mi}n}}
\end{equation}
for any $N = kn$. The above limits are not surprising: it was shown in \cite{Shuf} that the subalgebra $\CS_{m/n} \subset \CS$ that corresponds to the subalgebra $\CA_{m/n} \subset \CA$ under Theorem \ref{thm:shufhall} is precisely characterized by the existence and finiteness of the limits \eqref{eqn:copshuf}, by generalizing a result of \cite{FHHSY}. \\

\subsection{} 
\label{sub:bull}

Since $\Lambda$ is a polynomial ring in infinitely many variables, it has quite a large number of automorphisms. But if we require such automorphisms to also preserve the coproduct, then the only possibility for an automorphism is to independently rescale the power sum functions $p_k$. We will fix this ambiguity by introducing the multiplicative \textbf{norm} map:
$$
\ph : \Lambda \longrightarrow \BK, \qquad \qquad \ph(p_k) = 1
$$
As shown in \cite{Shuf}, the isomorphism \eqref{eqn:iso} makes the above map compatible with:
$$
\ph : \CS_{m/n} \longrightarrow \BK,  \qquad \ph (R) =  R(1,q,...,q^{N-1}) \cdot q^{-\frac {MN-M+N-k}2} \frac {(t-q)...(t-q^N)}{(1-q)...(1-q^N)} 
$$
where we write $M = km$ and $N = kn$. To summarize the above, we have the following implicit characterization, which uniquely determines the symmetric rational function $P_k^{m/n}$: \\


\begin{itemize}
	
\item it has $kn$ variables and homogenous degree $km$ \\

\item it satisfies the wheel conditions \eqref{eqn:wheel} \\

\item its degree in any number of $i\in (0,kn)$ of its variables is $ < \frac {mi}n$ \\

\item we have the normalization: $\ph\left(P_k^{m/n} \right) = 1$.\\


\end{itemize}

\subsection{} 
\label{sub:bullets}

We can generalize the above description to the following, true for all triples $(m,n,k) \in \BZ \times \BN \times \BN$ with $m$ and $n$ coprime: \\

\begin{lemma} 
\label{lem:uniqe}
	
For any $c \in \kk$ and sums \footnote{Our convention for the coproduct is Sweedler notation, which writes $\Delta(R) = R_1 \otimes R_2$ and implies a sum of tensors. Therefore, the complete notation should be $\Delta(R) = \sum_i R_1^{(i)} \otimes R_2^{(i)}$, although we will generally avoid this in order to not overburden notation} of tensors $R_{1,l} \otimes R_{2,l} \in \CS_{m/n} \otimes \CS_{m/n}$ for all $0<l<k$, there is at most one shuffle element $R$ such that: \\

\begin{itemize}
	
\item $R$ has $kn$ variables and homogenous degree $km$ \\
	
\item $R$ satisfies the wheel conditions \eqref{eqn:wheel} \\
	
\item the degree of $R$ in any number of $i\in (0,kn)$ of its variables is $ \leq \frac {mi}n$, with equality only if $i = ln$ for some $0<l<k$, in which case:
$$
\lim_{\xi \rightarrow \infty} \frac {R(\xi z_1,..., \xi z_{ln} \otimes z_{ln+1},..., z_{kn})}{\xi^{ml}} = R_{1,l} (z_1,...,z_{ln}) \otimes R_{2,l}(z_{ln+1},...,z_{kn})
$$
	
\item we have the normalization $\ph(R) = c$. 
	
\end{itemize}

\end{lemma}

\text{} \\
If we assume the existence of the isomorphism of bialgebras $\Lambda \stackrel{\cong}\rightarrow \CS_{m/n}$, then uniqueness becomes equivalent to the following statement on symmetric functions:
$$
\text{if }f,f'\text{ in }\Lambda\text{ are such that }\ph(f) = \ph(f') \text{ and}
$$
$$
\Delta(f) - f \otimes 1 - 1\otimes f = \Delta(f') - f' \otimes 1 - 1\otimes f' 
$$
then $f = f'$. This is clear, since the second condition implies that $f-f'$ is a constant multiple of the power sum function, and the first condition implies that this constant multiple is 0. At the end of this Section, we will give a direct proof of Lemma \ref{lem:uniqe}, which doesn't require the quite technical result that $\Lambda \cong \CS_{m/n}$. Our proof will be necessary for our computation in Section \ref{sec:stab}. \\

\subsection{} 
\label{sub:e}


The isomorphism:
$$
\Lambda \quad \stackrel{\cong}\longrightarrow \quad \CA_{m/n} \cong \CS_{m/n},
$$
$$
p_k \quad \longrightarrow \quad p_k^{m/n} \in \CA_{m/n} \cong \CS_{m/n} \ni P_k^{m/n} 
$$
will be referred to as the $\frac mn-$\textbf{plethysm}. The same terminology will be applied to the image of any symmetric polynomial $f\in \Lambda$ under the above isomorphism. Let us emphasize the fact that when one applies such a plethysm to a symmetric function, one obtains an operator on symmetric functions. Our main Theorem \ref{thm:stab} is concerned with the $\frac mn-$plethysm of elementary symmetric functions:
$$
e_k \quad \longrightarrow \quad e_k^{m/n} \in \CA_{m/n} \cong \CS_{m/n} \ni E_k^{m/n}
$$
According to the four bullets of Lemma \ref{lem:uniqe}, the shuffle element $E_k^{m/n}$ is uniquely characterized by the following four properties: \\

\begin{itemize}
	
\item $E_k^{m/n}$ has $kn$ variables and homogenous degree $km$ \\
	
\item $E_k^{m/n}$ satisfies the wheel conditions \eqref{eqn:wheel} \\
	
\item the degree of $E_k^{m/n}$ in any number of $i\in (0,kn)$ of its variables is $ \leq \frac {mi}n$, with equality only if $i = ln$ for some $0<l<k$, in which case:
$$
\lim_{\xi \rightarrow \infty} \frac {E_k^{m/n}(\xi z_1,..., \xi z_{ln} \otimes z_{ln+1},..., z_{kn})}{\xi^{ml}} = E_l^{m/n} (z_1,...,z_{ln}) \otimes E_{k-l}^{m/n} (z_{ln+1},...,z_{kn})
$$
\item we have the normalization $\ph(E_k^{m/n}) = \delta_k^1$. 
	
\end{itemize}

\text{} \\
Indeed, the third bullet is an immediate consequence of the fact that elementary symmetric functions satisfy the coproduct property $\Delta(e_k) = \sum_{l=0}^k e_l \otimes e_{k-l}$. \\


\subsection{} 

We leave the following as an exercise for the interested reader, since we will not use it in this paper. The proof follows the machinery developed in \cite{Shuf}. \\

\begin{exercise}
\label{prop:wadda}

For any coprime $(m,n) \in \BZ \times \BN$ and $k\in \BN$, the shuffle element $E_k^{m/n}$ that corresponds to $e_k$ under the isomorphism \eqref{eqn:iso} is given by $\quad E_{k}^{m/n} = $
$$
 = \left[ \frac {(1-t)(1-q)}{t-q} \right]^N \emph{Sym} \left[\frac {\prod_{i=1}^{N} z_i^{r_{\frac mn}(i)} \prod_{i=1}^{k-1} \left(1-\frac { z_{in}}{q t^{i-1} z_{in+1}} \right)}{[k]!_t \left(1-\frac {t z_1}{q z_2} \right)...\left(1-\frac {t z_{N-1}}{q z_{N}} \right)}  \prod_{1\leq i < j \leq N} \omega\left ( \frac {z_i}{z_j} \right) \right] \qquad  
$$
where we write $N = kn$, $[x]_t = t^{1-x}-t$ and $[x]!_t = [1]_t...[x]_t$. 
	
\end{exercise}

\text{} \\
We will however present give a new proof of Lemma \ref{lem:uniqe}, as this will be necessary for our proof of Theorem \ref{thm:stab}: \\

\begin{proof} \textbf{of Lemma \ref{lem:uniqe}:} Let us consider a shuffle element $R$ satisfying the four bullets, and prove that it is unique by obtaining recursive formulas for it.	For any composition $\rho =(\rho_1, \rho_2,..., \rho_l)$ of $kn$, let us write:
$$
R^\rho(y_1,...,y_l) = R(y_1,...,y_1q^{\rho_1-1},...,y_l,...,y_l q^{\rho_l-1}) = 
$$
\begin{equation}
\label{eqn:spec}	=\frac {r(y_1,...,y_1q^{\rho_1-1},...,y_l,...,y_l q^{\rho_l-1})}{\prod_{1\leq i < j \leq l} \prod_{a=1}^{\rho_i} \prod_{b=1}^{\rho_j} (y_j q^b - y_i q^{a-1}t )(y_jq^b - y_i q^{a+1}t^{-1})}
\end{equation}
where $r$ is the symmetric Laurent polynomial of \eqref{eqn:shufelem}. Note that $R^{(1,...,1)}$ is $R$ itself, while the first and fourth bullets give us:
$$
R^{(kn)}(y_1) =  y_1^{km} \cdot c
$$
Our approach will be to obtain formulas for $R^\rho$, inductively in decreasing lexicographic order of $\rho$. By \eqref{eqn:shuf}, $r$ in formula \eqref{eqn:spec} is a Laurent polynomial that satisfies the wheel conditions \eqref{eqn:wheel}. These conditions imply that $r$ vanishes when:
$$
	y_i q^{a+1} = y_j q^{b} t \qquad \text{for} \quad 1 < a \leq \rho_i, \quad 1 \leq b \leq \rho_j
	$$
	$$
	y_i q^{a} t = y_j q^{b+1} \qquad \text{for} \quad 1 \leq a < \rho_i, \quad 1 \leq b \leq \rho_j
	$$
	where we assume that $\rho_i \geq \rho_j$. If the inequality is the other way, then we simply change the roles of $i$ and $j$. The above zeroes are counted with the correct multiplicities, so the fraction in \eqref{eqn:spec} may be simplified to:
	\begin{equation}
	\label{eqn:pop1}
	R^\rho(y_1,...,y_l) =\frac {r^\rho(y_1,...,y_l)}{\prod_{1\leq i < j \leq l} \prod_{b=1}^{\rho_j} (y_j q^b - y_i t ) (y_j q^b t - y_i q^{\rho_i+1})}
	\end{equation}
	where $r^\rho(y_1,...,y_l)$ is a Laurent polynomial in $l$ variables. The problem with the above formula is that the denominator goes over all pairs $i < j$, but tacitly makes the assumption that $\rho_i \geq \rho_j$. We do not want to make this assumption, so we could either use a much more complicated way to label the indices in the above formula, or find a better way to write it. We choose the latter approach. For any choice of: 
	$$
	y_i = q^{x(\square)} t^{-y(\square)} \qquad \text{and} \qquad y_j = q^{x(\square')} t^{-y(\square')}
	$$ 
	the specializations \eqref{eqn:spec} correspond to the weights of two horizontal strips of lengths $\rho_i$ and $\rho_j$ which start at the boxes $\square$ and $\square'$, respectively. Let us consider all ways to translate one horizontal strip over the other such that they \textbf{partially overlap}, by which we mean that the resulting set of boxes can be divided into two horizontal strips of lengths:
	\begin{equation}
	\label{eqn:greater}
	\max(\rho_i,\rho_j) + b \qquad \text{and} \qquad \min(\rho_i,\rho_j) - b, 
	\end{equation}
	for some $b>0$. Figure 3 below shows a certain example of partial overlapping: \\
	
	\begin{picture}(350,130)(0,45)
	\label{fig2}
	
	\put(90,180){\line(1,0){80}}
	\put(90,160){\line(1,0){80}}
	\put(90,180){\line(0,-1){20}}
	\put(110,180){\line(0,-1){20}}
	\put(130,180){\line(0,-1){20}}
	\put(150,180){\line(0,-1){20}}
	\put(170,180){\line(0,-1){20}}
	
	\put(97,168){$y_i$}
	
	\put(120,140){\line(1,0){120}}
	\put(120,120){\line(1,0){120}}
	\put(120,140){\line(0,-1){20}}
	\put(140,140){\line(0,-1){20}}
	\put(160,140){\line(0,-1){20}}
	\put(180,140){\line(0,-1){20}}
	\put(200,140){\line(0,-1){20}}
	\put(220,140){\line(0,-1){20}}
	\put(240,140){\line(0,-1){20}}
	
	\put(127,128){$y_j$}
	
	\put(90,115){\vector(-1,-2){15}}
	\put(250,115){\vector(1,-2){15}}
	
	\put(0,75){\line(1,0){80}}
	\put(0,55){\line(1,0){80}}
	\put(0,75){\line(0,-1){20}}
	\put(20,75){\line(0,-1){20}}
	\put(40,75){\line(0,-1){20}}
	\put(60,75){\line(0,-1){20}}
	\put(80,75){\line(0,-1){20}}
	
	\put(1,63){$y_j / q$}
	
	\put(22,73){\line(1,0){120}}
	\put(22,53){\line(1,0){120}}
	\put(22,73){\line(0,-1){20}}
	\put(42,73){\line(0,-1){20}}
	\put(62,73){\line(0,-1){20}}
	\put(82,73){\line(0,-1){20}}
	\put(102,73){\line(0,-1){20}}
	\put(122,73){\line(0,-1){20}}
	\put(142,73){\line(0,-1){20}}
	
	\put(28,62){$y_j$}
	
	\put(260,75){\line(1,0){80}}
	\put(260,55){\line(1,0){80}}
	\put(260,75){\line(0,-1){20}}
	\put(280,75){\line(0,-1){20}}
	\put(300,75){\line(0,-1){20}}
	\put(320,75){\line(0,-1){20}}
	\put(340,75){\line(0,-1){20}}
	
	\put(261,63){$y_j q^5$}
	
	\put(162,73){\line(1,0){120}}
	\put(162,53){\line(1,0){120}}
	\put(162,73){\line(0,-1){20}}
	\put(182,73){\line(0,-1){20}}
	\put(202,73){\line(0,-1){20}}
	\put(222,73){\line(0,-1){20}}
	\put(242,73){\line(0,-1){20}}
	\put(262,73){\line(0,-1){20}}
	\put(282,73){\line(0,-1){20}}
	
	\put(168,62){$y_j$}
	
	\put(150,37){\mbox{Figure 3}}
	
	\end{picture}
	
	\textbf{} \\
	In the above, we have $\rho_i=4$ and $\rho_j=6$. In the partial overlap on the left we have $b=1$, whereas in the one on the right we have $b=3$. We also allow the smallest of the resulting strips to have length 0. There are precisely $2\min(\rho_i,\rho_j)$ such translations, and they correspond to the terms which appear in the denominator of formula \eqref{eqn:pop1}. So we can rewrite this formula as:
	\begin{equation}
	\label{eqn:pop2}
	R^\rho(y_1,...,y_l) =\frac {r^\rho(y_1,...,y_l)}{\prod_{1\leq i < j \leq l} \prod_{b \in S^\pm_{ij}} (y_j q^b - y_i t^{\pm 1})}
	\end{equation}
	where the set $S^+_{ij}$ (respectively $S^-_{ij}$) consists of all integers $b$ which make the specialization $y_i = y_jq^b$ correspond to a translation that makes the $j-$th horizontal strip partially overlap the $i-$th horizontal strip, with the former (respectively latter) being to the left of the other. The two instances of partial overlap in Figure \ref{fig2} correspond to $y_i=y_j/q$ for the example on the left, and $y_i=y_jq^5$ for the example on the right. Note that the cardinality of each of the sets $S^\pm_{ij}$ is equal to $\min(\rho_i,\rho_j)$. The advantage of \eqref{eqn:pop2} is that it makes no assumptions about the relative sizes of $\rho_i$ and $\rho_j$, because they are encoded in the definition of the sets $S^\pm_{ij}$. 
	
	\textbf{} \\
	We will now obtain inductive formulas for the Laurent polynomial $r^\rho$ in decreasing lexicographic order of $\rho$, which will prove that this Laurent polynomial is unique. Because of the degree restrictions in the third bullet of Subsection \ref{sub:bullets}, we have:
	\begin{equation}
	\label{eqn:degrestr}
	\text{total }\deg \ r^\rho = m + 2\sum_{1\leq i < j \leq l} \min(\rho_i,\rho_j)
	\end{equation}
	and:
	$$
	\frac {m \rho_i}n  \leq \deg_{y_i} r^\rho \leq \frac {m \rho_i}n + 2\sum_{1\leq i \neq j \leq l} \min(\rho_i,\rho_j) 
	$$
	in each variable $y_i$. The space of Laurent polynomials satisfying the above degree conditions is quite large, but it is constrained by the following recurrence relations:
	\begin{equation}
	\label{eqn:cond}
	R^\rho(y_1,...,y_l)|_{y_i = y_j q^b} = R^{\rho(i \stackrel{b}\leftrightarrow j)}(y_1,...,y_l)|_{y_i = y_j q^b}
	\end{equation}
	for any $i<j$ and any $b\in S^\pm_{ij}$, where $\rho(i \stackrel{b}\leftrightarrow j)$ denotes the composition obtained from $\rho$ by replacing $\rho_i$ and $\rho_j$ by the two numbers in \eqref{eqn:greater}. Note this new composition is strictly greater than $\rho$ in lexicographic order. Moreover, the specializations in \eqref{eqn:cond} precisely correspond to the translations of horizontal strips which go into the definition of the sets $S^\pm_{ij}$ (see Figure 3 for a depiction).
	
	\textbf{} \\
	If we regard \eqref{eqn:cond} as an identity of rational functions in the single variable $y_l$, note that we have $2\sum_{i<l} \min(\rho_i,\rho_l)$ conditions that involve this variable. The degree constraints \eqref{eqn:degrestr} imply that $r^\rho$ only has terms $y_l^d$ for: 
	$$
	\frac {m\rho_l}n \leq d \leq \frac {m\rho_l}n +  2\sum_{i \neq l} \min(\rho_i,\rho_l) 
	$$ 
	Note that there are as many such $d$ as there are conditions \eqref{eqn:cond}, unless $n|\rho_l$, when we have one extra $d$ that can be determined by specifying the least order term in $y_l$. Thus, we can apply Lagrange polynomial interpolation:
	$$
	r^\rho(y_1,...,y_l) = \text{least order term} \cdot \prod_{1\leq i < l} \prod_{a \in S^\pm_{il}} \left(1 - \frac {y_l q^a}{y_i}\right) +
	$$
	$$
	+  \sum_{1\leq i < l} \sum_{a \in S^\pm_{il}}   \left[ \left(\frac {y_lq^a}{y_i} \right)^{\left \lfloor \frac {m\rho_l}n \right \rfloor+1}  r^{\rho(i \stackrel{a}\leftrightarrow j)}(y_1,...,y_l)|_{y_l q^a = y_i} \prod_{j<l}^* \prod^*_{b \in S^\pm_{jl}} \frac {y_j - y_l q^b}{y_j - y_i q^{b-a}} \right]
	$$
	where $\prod^*$ means that we exclude the linear factor which vanishes in the denominator (and also exclude the corresponding factor in the numerator). The first term is only thought to exist if $n|\rho_l$. To use the above recurrence, we need to translate this information in terms of the rational functions $R_k^\rho$:
	$$
	R^\rho(y_1,...,y_l) = \gamma \cdot R_{1, k-\frac {\rho_l}n}^{(\rho_1,...,\rho_{l-1})}(y_1,...,y_{l-1}) \cdot R_{2,\frac {\rho_l}n}^{(\rho_l)}(y_l)  + 
	$$
	\begin{equation}
	\label{eqn:lagrange}
	+ \sum^{i<l}_{a \in S^\pm_{il}} \gamma^{\pm}_{i,a} \cdot R^{\rho(i \stackrel{a}\leftrightarrow l)}(y_1,...,y_l)|_{y_l q^a = y_i} \qquad 
	\end{equation}
	where the first summand corresponds to the least order term in the variable $y_l$, which is specified by the limit in the third bullet of Subsection \ref{sub:bullets}, and only exists if $n|\rho_l$. The coefficients in the above are given by:
	\begin{equation}
	\label{eqn:gamma}
	\gamma = \prod^{i<l}_{a \in S^\pm_{il}}  \frac {1 - \frac {y_l q^a}{y_i}}{1 - \frac {y_lq^a}{y_i t^{\pm 1}}}
	\end{equation}
	$$
	\gamma^{\pm}_{i,a} =  \left( \frac {y_l q^a}{y_i} \right)^{\left \lfloor \frac {m\rho_l}n \right \rfloor + 1} \prod^{* j<l}_{b \in S_{jl}^{\pm'}} \frac {1 - \frac {y_l q^b}{y_j}}{1 - \frac {y_iq^{b-a}}{y_j}} \prod^{j<l}_{b \in S_{jl}^{\pm'}} \frac {1 - \frac {y_iq^{b-a}}{y_jt^{\pm' 1}}}{1 - \frac {y_l q^b}{y_j t^{\pm'1}}}
	$$
	Recall that the $*$ above the product signs mean that we discard the factor which vanishes from the denominator, and also the corresponding factor from the numerator. The above formula concludes the proof, because it shows how the rational function $R = R^{(1,...,1)}$ can be reconstructed inductively from the constant $c \sim R^{(kn)}$ and the various $R_{1,l}$ and $R_{2,l}$. 
	
\end{proof}

\subsection{} Lemma \ref{lem:uniqe} has a number of consequences, such as:
\begin{equation}
\label{eqn:inverse}
P_k^{m/n} \left( z_1^{-1},...,z_{kn}^{-1} \right) = P_k^{-m/n}(z_1,...,z_{kn})
\end{equation}
Indeed, this follows from the fact that the four bullets of Subsection \ref{sub:bull} still hold when we replace $z_i \rightarrow z_i^{-1}$ and $m \rightarrow -m$. Along the same lines, one sees that the composition of two isomorphisms \eqref{eqn:iso}:
$$
\CS_{m/n} \stackrel{\cong}\longrightarrow \Lambda \stackrel{\cong}\longrightarrow \CS_{-m/n}
$$
is simply given by $R(z_1,...,z_N) \longrightarrow R(z_1^{-1},...,z_N^{-1})$. In particular, the subalgebra $\CS_0 \subset \CS$ is invariant under inverting all the variables. In the Appendix, we will use the proof of Lemma \ref{lem:uniqe} to obtain the following computation: \\


\begin{proposition}
	\label{prop:eval}
	
	For any $l \leq kn$ and the partition $\mu_l = \underbrace{(nk-l+1,1,...,1)}_{l\text{ terms}}$:
	$$
	P_k^{m/n}(\mu_l) = \sum_{i=0}^{n-1} q^{\sum_{j=1}^{nk-l} \left \lfloor \frac {mj+i}n \right \rfloor} t^{-\sum_{j=1}^{l-1} \left \lceil \frac {mj-i}n \right \rceil} \cdot
	$$
	\begin{equation}
	\label{eqn:eval}
	\cdot \frac {1-q^{-k}}{1-tq^{-1}} \prod_{i=1}^{nk-l} \frac {1-q^{-i}}{1-q^{-i-1}t} \prod_{i=1}^{l-1} \frac {1-t^i}{1-t^{i+1}q^{-1}}
	\end{equation}
	where the evaluation of a shuffle element at a partition is defined in \eqref{eqn:well}. We assume $\gcd(m,n) = 1$, otherwise the greatest common divisor can be absorbed in $k$. \\ 
	
	
\end{proposition}


\section{Stable bases}
\label{sec:stab}

\subsection{} 
\label{sub:stab}

We will now describe a construction of Maulik and Okounkov (\cite{MO2}), known as the \textbf{stable basis}. This construction is geometric and very general, but we will simply spell out the definition in our particular setup. For any $\frac mn \in \BQ$, Maulik-Okounkov prove the existence of a unique integral \footnote{A symmetric function is integral if it expands in terms of Schur functions with coefficients in $\BZ[q^{\pm 1}, t^{\pm 1}]$. A basis is integral if it consists entirely of integral functions} basis $\{s_\lambda^{m/n}\}_{\lambda \text{ partition}}$ of $\Lambda$:
$$
s_\lambda^{m/n}  = \sum_{\mu \unlhd \lambda} M_\mu \cdot c^\mu_\lambda(q,t) 
$$
where $M_\mu$ are the renormalized Macdonald polynomials \eqref{eqn:renormalization}, such that:
\begin{equation}
\label{eqn:normal}
c_\lambda^\lambda(q,t) = \prod_{\square \in \lambda} \left(t^{-l(\square)} - q^{a(\square)+1} \right)
\end{equation}
and $c^\mu_\lambda \in \BZ[q^{\pm 1}, t^{\pm 1}]$ are such that for all $\mu \lhd \lambda$ in the dominance ordering: 
\begin{equation}
\label{eqn:lim1}
\degu c^\mu_\lambda < \frac mn(o_{\mu}-o_\lambda) + {\max}_{\mu}
\end{equation}
\begin{equation}
\label{eqn:lim2}
\degd c^\mu_\lambda \geq \frac mn(o_{\mu}-o_\lambda) + {\min}_{\mu}
\end{equation}
where we define the \textbf{upper degree} of a Laurent polynomial $c(q,t)$ as:
\begin{equation}
\label{eqn:ord1}
\degu c(q,t) = \text{ order of }c (az,bz) \text{ as } z\longrightarrow \infty 
\end{equation}
and the \textbf{lower degree} as:
\begin{equation}
\label{eqn:ord2}
\degd c(q,t) = \text{ order of }c (az,bz) \text{ as } z\longrightarrow 0 
\end{equation}
\footnote{For example, the function $c(q,t) = q - t + t^{-1}q^{-1}$ has $\degu c = 1$ and $\degd c = -2$} and set:
$$
o_\lambda = \sum_{\square \in \lambda} o_\square \qquad \qquad {\min}_\lambda =  - \sum_{\square \in \lambda} l(\square) \qquad \qquad {\max}_\lambda = |\lambda| + \sum_{\square \in \lambda} a(\square) 
$$
\footnote{The common combinatorial notation is $\min_\lambda = -n(\lambda)$ and $\max_\lambda = |\lambda|+n'(\lambda)$. We will retain the notations $\min$ and $\max$ because the stable basis construction is more general than our setup} From now on, the phrase ``term of highest/lowest degree" of a Laurent polynomial $c(q,t)$ will refer to the sum of monomials in $q,t$ for which the order in \eqref{eqn:ord1}/\eqref{eqn:ord2} is attained. For example, we have:
$$
\hd \left(q - t + t^{-1}q^{-1} \right)= q-t, \qquad \qquad \ld \left(q - t + t^{-1}q^{-1} \right) = t^{-1}q^{-1}
$$

\text{}

\subsection{} The above construction arises from the geometry of the Hilbert scheme, whose equivariant $K-$theory groups are identified with $\Lambda$, and whose fixed points are identified with partitions $\lambda$. Then the stable basis is a distinguished basis of $K-$theory, which is inductively built from the stable leaves of a one-dimensional torus action. The polynomial \eqref{eqn:normal} is simply the unstable part of the tangent space at the fixed point $\lambda$. The rational number $\frac mn$ refers to a rational multiple of the line bundle $\CO(1)$ on the Hilbert scheme, which appears in the $K-$theoretic stable basis construction.

\text{} \\
Let us explain conditions \eqref{eqn:lim1} and \eqref{eqn:lim2}: the Newton polygon of $c_\lambda^\mu(q,t)$ is required to lie in a certain diagonal strip in the plane, and the only case when the southeast side of the strip is allowed to be touched by the Newton polygon is when $\mu = \lambda$. The notations $\nwarrow$ and $\searrow$ refer to the northwest and southeast directions in the strip, and \eqref{eqn:ord1}-\eqref{eqn:ord2} pick up the vertex of the Newton polygon which is extremal in the direction of the arrow. Although we will neither need nor prove this statement, it is expected that $s_\lambda^0$ given by the above definition coincides with the usual Schur function $s_\lambda$. Moreover, note that it is immediate from the definition that:
$$
s_\lambda^{r+1} = \frac {\nabla s_\lambda^r}{\prod_{\square \in \lambda} \chi_\square}
$$
for any rational number $r$. Finally, when $r=\infty$, note that conditions \eqref{eqn:lim1} and \eqref{eqn:lim2} boil down to requiring that $c^\mu_\lambda = 0$ for $\mu \lhd \lambda$. Therefore, $s_\lambda^\infty = P_\lambda$ are the usual Macdonald polynomials. However, these are not integral, and so do not fit our definition precisely. Thus we will think of the $\infty$ case as being special. \\


	
	

\subsection{} Our main Theorem \ref{thm:stab} involves showing how the operator $\oe_k^{m/n}$ acts on the stable basis, for any $k,m,n$ with $n>0$ and $\gcd(m,n)=1$. To obtain the required formula, we need to compute bounds on the upper and lower degrees of the matrix coefficients of this operator in the basis $M_\lambda$. By formula \eqref{eqn:shuf}, we need to bound the upper and lower degrees of the expressions:
$$
E_k^{m/n}(\lamu)
$$ 
as $\lamu$ goes over all skew diagrams of $N = kn$ boxes, where $E_k^{m/n}$ is the unique shuffle element which is characterized by the four bullets in Subsection \ref{sub:e}. For any diagram $\lamu$, let us write:
$$
\#_{\lamu} = \text{the number of pairs of boxes in }\lamu\text{ of the same content}
$$
Then we will prove the following: \\

\begin{proposition} 
\label{prop:main}

For any skew diagram $\lambda \backslash \mu$ of $N = kn$ boxes, we have:
\begin{equation}
\label{eqn:relation1}
\degu E_k^{m/n}(\lambda \backslash \mu) \leq \frac mn (o_\lambda - o_\mu) + \#_{\lamu} + \frac {k(n-1)}2
\end{equation}

\begin{equation}
\label{eqn:analog}
\degd E_k^{m/n}(\lambda \backslash \mu) \geq \frac mn (o_\lambda - o_\mu) - \#_{\lamu} - \frac {k(n+1)}2
\end{equation}
In the first relation, we have equality if and only if $\lamu$ is a vertical $k-$strip of $n-$ribbons $B_1,...,B_k$ (see Subsection \ref{sub:ribbon} for the definition), in which case the term of highest degree \footnote{That is, the sum of monomials which produces the top degree term in \eqref{eqn:ord1}} is:
\begin{equation}
\label{eqn:lana}
\emph{h.d. } E_k^{m/n}(\lambda \backslash \mu) = \left(\frac qt\right)^{\flat_{\lamu}} t^{\#_\lamu} \cdot
\end{equation}
$$
\prod_{i=1}^k \prod_{j=1}^n \chi_{j}(B_i)^{\left \lfloor \frac {mj}n \right \rfloor - \left \lfloor \frac {m(j-1)}n \right \rfloor} \cdot \frac {\prod^{\blacksquare \in \lamu, \ o_\blacksquare = o_\square}_{\square \emph{ outer corner of }\lamu}  \left(\frac {q \chi_\blacksquare}{t\chi_\square} - 1\right)}{\prod^{\blacksquare \in \lamu, \ o_\blacksquare = o_\square}_{\square \emph{ inner corner of }\lamu} \left(\frac {q \chi_\blacksquare}{t\chi_\square} - 1 \right)}
$$
where $\flat_{\lamu}$ denotes the number of boxes $\blacksquare \in \lamu$ which are precisely $p$ diagonal steps away from an inner corner or vertical inner boundary of $\lamu$, counted with multiplicity $p+1$, and $\chi_{j}(B_i)$ denotes the content of the $j-$th box in the ribbon $B_i$. \\

\end{proposition}

\begin{proof} Let us prove relations \eqref{eqn:relation1} and \eqref{eqn:lana} on $\degu$ and leave the analogous \eqref{eqn:analog} as an exercise. We will write: 
$$
R_k = E_k^{m/n}(z_1,...,z_N)
$$
with $N=kn$ and use the notations in the proof of Lemma \ref{lem:uniqe}. In particular, we will prove the following slightly more general statement. \\

\begin{claim} 
\label{claim}

Let us consider a union $D = \{D_1,...,D_l\}$ of horizontal strips $D_i$ lying in the first quadrant. Then we have:
\begin{equation}
\label{eqn:bound}
\degu R_k(D) \leq \frac mn \cdot o_D + \#_D + \frac {k(n-1)}2
\end{equation}
by which we mean that we apply the rational function $R_k$ to the set of weights of the boxes in $D$. We write $o_D = \sum_{\square \in D} o_\square$ and $\#_D$ stands for the number of pairs of boxes of $D$ with the same content. 

\end{claim}

\textbf{} \\ 
Let $\rho = (\rho_1, ..., \rho_l)$ denote the composition of $N=kn$ determined by the lengths of our horizontal strips. We will prove the above claim by induction in increasing order of $k$, and then in decreasing lexicographic order of $\rho$. The base case, namely when $\rho = (N)$, is dealt with by the fourth bullet in Subsection \ref{sub:e}. The quantity whose upper degree we need to bound is:
\begin{equation}
\label{eqn:above}
R_k(D) = R_k^\rho(y_1,...,y_l)
\end{equation}
where we set $y_i$ to be the weight of the leftmost box of the $i-$th strip, and we denote by $o_i$ the content of the box of weight $y_i$. The above quantity \eqref{eqn:above} is a certain evaluation of the LHS of \eqref{eqn:lagrange}, so we may replace it by the corresponding evaluation of the RHS of \eqref{eqn:lagrange}. We may apply the induction hypothesis to the terms that appear in the RHS:
$$
\degu R_{k-1}^{\rho'}(y_1,...,y_{l-1}) \quad \leq \quad \frac mn \cdot o_{D'} + \#' + \frac {(k-1)(n-1)}2 
$$

$$
\degu R_k^{\rho(i \stackrel{a}\leftrightarrow l)}(y_1,...,y_l)|_{y_l q^a \rightarrow y_i} \leq \frac mn \cdot o_{D^{\pm}_{i,a}} + \#^{\pm}_{i,a} + \frac {k(n-1)}2 
$$
where $\#'$ and $\#^{\pm}_{i,a}$ denote the number of pairs of boxes of the same content among the arguments of each $R$ in the LHS, computed with respect to the boxes in the sets $D'$ and $D^{\pm}_{i,a}$. These latter sets of boxes are what is obtained from the set $D$ by removing the $l-$th strip, respectively by performing the substitution $y_l q^a \rightarrow y_i$ (which amounts to translating the entire $l-$th horizontal strip to start at the box $a$ units to the left of the start of the $i-$th horizontal strip). As for the coefficients $\gamma$ and $\gamma^{\pm}_{i,a}$, these are just products of linear factors, and we claim that:
\begin{equation}
\label{eqn:ex1}
\degu \gamma = \#_D - \#' + \frac mn(o_D-o_{D'}) + \frac {n-1}2
\end{equation}

\begin{equation}
\label{eqn:ex2}
\degu \gamma^{\pm}_{i,a} \leq \#_D - \#^{\pm}_{i,a} + \frac mn(o_D-o_{D_{i,a}^{\pm}})
\end{equation}
Once we prove the above claims, we will have proved the induction step that establishes Claim \ref{claim}. So let us show how to prove the more difficult equality \eqref{eqn:ex2} and leave the first as an analogous exercise. According to formula \eqref{eqn:gamma}, the coefficient $\gamma^{\pm}_{i,a}$ consists of three parts:
$$
\degu  \left( \frac {y_l q^a}{y_i} \right)^{\left \lfloor \frac {m\rho_l}n \right \rfloor + 1} = (o_l+a-o_i) \left( \left \lfloor \frac {m\rho_l}n \right \rfloor + 1 \right)
$$
$$
\degu \prod^{*  j<l}_{b \in S_{jl}^{\pm'}} \frac {1 - \frac {y_l q^b}{y_j}}{1 - \frac {y_iq^{b-a}}{y_j}}  = - M(o_l+a-o_i) + \sum^{j<l}_{b \in S_{jl}^{\pm'}} M(o_l+b-o_j) - M(o_i+b-a-o_j)
$$
\begin{equation}
\label{eqn:swarm}
\degu \prod^{j<l}_{b \in S_{jl}^{\pm'}} \frac {1 - \frac {y_iq^{b-a}}{y_jt^{\pm' 1}}}{1 - \frac {y_l q^b}{y_j t^{\pm'1}}} = \sum^{j<l}_{b \in S_{jl}^{\pm'}} M(o_i+b-a-o_j \mp' 1) - M(o_l+b-o_j \mp' 1) \qquad \quad
\end{equation}
where we write $M(z) = \max(z,0)$. The first equation is clear, so let's focus on the last two. The first term in the RHS of the second equation above comes from the $*$ in the products: there is a factor missing. Adding the three equalities together gives us:
$$
\degu \gamma^{\pm}_{i,a} = (o_l+a-o_i) \left( \left \lfloor \frac {m\rho_l}n \right \rfloor + 1 \right) - M(o_l+a-o_i) + \sum^{j<l}_{b \in S_{jl}^{\pm'}}
$$
$$
\left[ M(o_l+b-o_j) - M(o_i+b-a-o_j) + M(o_i+b-a-o_j \mp' 1) - M(o_l+b-o_j \mp' 1) \right] 
$$

$$
\leq \frac {m\rho_l}n (o_l-o_i+a) + \sum^{j<l}_{b \in S_{jl}^{\pm'}} \left[ \delta_{\pm'(o_l-o_j+b)>0} - \delta_{\pm'(o_i-o_j+b-a)>0}  \right] = \frac mn(o_D-o_{D_{i,a}^{\pm}}) + \#_D - \#^{\pm}_{i,a}
$$
where the last equality holds simply by unwinding the definition of $S_{jl}^{\pm'}$. One observes that the above inequality becomes an equality if and only if either: \\

\begin{itemize} 
	
\item $o_l+a=o_i$, so the corresponding translation in $S_{il}^\pm$ comes from sliding $D_l$ diagonally in the southwest-northeast direction \\

\item $o_l+a > o_i$ and $n|\rho_l$, so the corresponding translation is $S_{il}^\pm$ comes from sliding $D_l$ diagonally and then to the left. 

\end{itemize}

\text{} \\
Because $\gcd(m,n)=1$, the second bullet above would imply $\rho_l > n$. Since successively sliding the strip $D_l$ would only increase its length, the fourth bullet of Subsection \ref{sub:e} would eventually produce an answer of 0 when we evaluate $R_k$ at the set of weights of the boxes in $D$. Therefore, to produce a non-zero term we can only successively iterate diagonal translations as in the first bullet. In other words, the equality in \eqref{eqn:bound} can only be attained if we can perform a succession of moves of the following kinds: \\

\begin{enumerate}

\item remove an entire horizontal strip of length $n$ (this corresponds to \eqref{eqn:ex1}) \\

\item slide a horizontal strip diagonally until it partially overlaps with another horizontal strip (this corresponds to \eqref{eqn:ex2})

\end{enumerate}
 
\textbf{} \\
and manage to eliminate the entire set of boxes $D$ without ever obtaining a horizontal strip of length $>n$. We are free to choose which strips to remove/slide at each step, and when $D = \lamu$ is a skew Young diagram there is a preferred choice which will make all but one of the coefficients $\gamma^{\pm}_{i,a}$ vanish. 

\textbf{} \\
The thing to do is to let $D_l$ be the topmost row of $\lamu$. If we slide it diagonally onto a row $D_i$ that is more than one step down, then the corresponding coefficient $\gamma^{+}_{i,a}$ vanishes because of the numerator of \eqref{eqn:gamma}. Moreover, translations in $S^-_{jl}$ cannot be diagonal slides, and hence do not contribute anything to our computation. So the only slide that produces a non-zero coefficient is precisely one row down, thus producing two overlapping strips. Take the longest of these strips: again, if we slide it more than one row down, the corresponding coefficient $\gamma^{+}_{i,a}$ vanishes because of the numerator of \eqref{eqn:gamma}. Repeating this argument, we see that the only way to obtain a non-zero coefficient is by repeatedly sliding one row at a time and then removing it whenever it gains length $n$. This is precisely the game in Subsection \ref{sub:game}, and we proved in Lemma \ref{lem:game} that the game ends with all boxes removed, thus producing equality in \eqref{eqn:bound}, only if $\lamu$ can be covered by a vertical $k-$strip of $n-$ribbons.

\textbf{} \\
In this case, we need only evaluate the term of highest degree in $R(\lamu)$, and we will do so by using the same formula \eqref{eqn:lagrange}. Letting $D=\lamu$ and writing $\rho_1,...,\rho_l$ for the lengths of the rows of $D$ ordered bottom to top, we have:
$$
\hd R_k(D) = \left( \hd
 R_k(D') \right)\cdot \left( \hd \gamma^{+}_{i,a} \right) = \left( \hd R_k(D') \right) \cdot \left(\frac qt\right)^{\left \lfloor \frac {m\rho_l}n \right \rfloor} \cdot
$$

$$
\prod^{o_l+b=o_j}_{j\neq l, \ b \in S_{jl}^+} \frac {1 - \frac {y_lq^b}{y_j}}{1 - \frac {y_iq^{b-a}}{y_j}} \prod^{o_l+b=o_j+1}_{j\neq l, \ b \in S_{jl}^+} \frac {1 - \frac {y_iq^{b-a}}{y_j t}}{1 - \frac {y_lq^b}{y_j t}} \prod^{o_l + b = o_j + 1}_{j\neq l, \ b \in S_{jl}^+} \frac qt
$$
where $D'$ is the multiset of boxes obtained by sliding the first row of $D$ diagonally down onto its second row. We continue sliding the top row as many times as the height of the outer ribbon $B_\out$ of $\lamu$. As explained in Subsection \ref{sub:game}, at the end of this procedure we will have straightened the outer ribbon into a horizontal strip of length $n$. When we remove this horizontal strip, we have:
$$
\hd R_k(D) = \left( \hd R_k(D \backslash B_\out) \right) \cdot \left(\hd \gamma \right) \left(\frac qt\right)^{*} \cdot
$$

$$
\cdot \left(\frac qt\right)^{\flat_\out}  \prod^{\blacksquare \in B_\out, o_\blacksquare = o_\square - 1}_{\square \text{ start of some row}} \frac  {\frac {q \chi_\blacksquare}{\chi_\square} - 1}{\frac qt - 1} \prod^{\blacksquare \in B_\out, o_\blacksquare = o_\square}_{\square \text{ start of some row}} \frac {\frac qt - 1}{\frac {q \chi_\blacksquare}{t \chi_\square} - 1}
$$
where $\flat_\out$ denotes the number of boxes $\blacksquare \in B_\out$ which are precisely $p$ steps diagonally from an inner corner or vertical inner boundary of $\lamu$, counted with multiplicity $p+1$. The exponent $*$ of $q/t$ on the first line precisely amounts to the ratio between the quantity: 
$$
\prod_{i=1}^n \chi_i(B_\out)^{\left \lfloor \frac {mi}n \right \rfloor - \left \lfloor \frac {m(i-1)}n \right \rfloor}
$$
before and after bubbling down the outer ribbon. We will now use the definition of $\gamma$ in \eqref{eqn:gamma}, and also cancel like factors in the second line above:
$$
\hd R_k(D) = \left( \hd R_k(D \backslash B_\out) \right) t^{|\blacksquare \in B_\out \text{ on the same diagonal as } \square \in D\backslash B_\out|}
$$
$$
\left(\frac qt\right)^{\flat_\out} \prod_{i=1}^n \chi_i(B_\out)^{\left \lfloor \frac {mi}n \right \rfloor - \left \lfloor \frac {m(i-1)}n \right \rfloor} \cdot \frac {\prod^{\blacksquare \in B_\out, \ o_\blacksquare = o_\square}_{\square \text{ outer corner of }\mu}  \left(\frac {q \chi_\blacksquare}{t\chi_\square} - 1 \right)}{\prod^{\blacksquare \in B_\out, \ o_\blacksquare = o_\square}_{\square \text{ inner corner of }\mu} \left(\frac {q \chi_\blacksquare}{t\chi_\square} - 1\right)}
$$
Iterating the above for the skew diagram $D\backslash B_\out$, which is covered by a vertical strip of one less $n-$ribbons than $D$, we obtain precisely \eqref{eqn:lana}.



\end{proof}

\subsection{} We will now use the previous result to obtain Theorem \ref{thm:stab}. A very important aspect of the proof is the fact that the operators $e_k^{m/n}$ are integral, i.e.:
\begin{equation}
\label{eqn:integrality}
e_k^{m/n} \cdot s_\mu^{m/n} = \sum_\lambda s_\lambda^{m/n} \cdot b^\lambda(q,t)
\end{equation}
for some Laurent polynomials $b^\lambda(q,t) \in \BZ[q^{\pm 1}, t^{\pm 1}]$. This condition can be weakened to allow $b^\lambda$ to have denominators arbitrary polynomials in $q/t$ and the subsequent proof would still apply, but unfortunately even this weakened statement  cannot be obtained with the tools developed in the present paper. In fact, the proof of \eqref{eqn:integrality} is geometric. The stable basis elements $s_\lambda^{m/n}$ are $K-$theory classes supported on certain Lagrangian subvarieties $L_\lambda$ of Hilbert schemes, and $e_k^{m/n}$ is given by a certain vector bundle on the so-called ``flag Hilbert scheme" studied in \cite{N}. Because the flag Hilbert scheme is a Lagrangian correspondence, it takes any $K-$theory class supported on the subvariety $L_\mu$ to an integral combination of $K-$theory classes supported on the subvarieties $L_\lambda$. \\

\begin{proof} \textbf{of Theorem \ref{thm:stab}:} Let us consider the class:
$$
\sigma = e_k^{m/n} \cdot s_\mu^{m/n} = e_k^{m/n} \cdot \sum_{\nu \leq \mu} M_\nu c_\mu^\nu(q,t)
$$
for some positive integer $k$ and some partition $\mu$. We may compute the RHS using formula \eqref{eqn:shuf}:
$$
\sigma = \sum_{\lambda} M_\lambda \cdot \sum_{\nu \unlhd \mu}  c_\mu^\nu(q,t) E_k^{m/n}(\lambda \backslash \nu) \prod_{\blacksquare \in \lambda \backslash \nu} \left[ \left(t - q \chi_\blacksquare \right) \prod_{\square \in \nu} \omega \left(\frac {\chi_{\blacksquare}}{\chi_{\square}} \right) \right] =: \sum_{\lambda} M_\lambda \cdot d^\lambda
$$
We want to prove that $\sigma$ equals the RHS of \eqref{eqn:pieri}. \\

\begin{claim}

To complete the proof of Theorem \ref{thm:stab}, it is enough to prove the following assertions about the coefficients $d^\lambda$ of the above expression:
\begin{equation}
\label{eqn:inequality0}
\degu d^\lambda \leq \frac mn(o_\lambda - o_\mu)+ {\max}_\lambda + \frac {k(n-1)}2
\end{equation}

\begin{equation}
\label{eqn:inequality}
\degd d^\lambda \geq \frac mn(o_\lambda - o_\mu)+ {\min}_\lambda + \frac {k(n-1)}2
\end{equation}
and that the first inequality becomes an equality precisely when $\lambda \backslash \mu$ can be covered by a $k-$strip of $n-$ribbons $B_1,...,B_k$, with the term of highest degree equal to the monomial:
\begin{equation}
\label{eqn:lowestorder}
\emph{h.d. } d^\lambda = (-1)^{|\lambda|} q^{\max_\lambda} (-1)^{\emph{ht }} \prod_{i=1}^k \prod_{j=1}^n \chi_{j}(B_i)^{\left \lfloor \frac {mj}n \right \rfloor - \left \lfloor \frac {m(j-1)}n \right \rfloor}
\end{equation}

\end{claim}

\begin{proof} Let us assume the conclusion of the claim, and consider the expansion of $\sigma$ in terms of the stable basis, as provided by \eqref{eqn:integrality}: 
$$
\sum_\lambda M_\lambda \cdot d^\lambda =  \sum_\lambda s_\lambda^{m/n} \cdot b^\lambda
$$
We need to prove that $b^\lambda = (-1)^{\high} \prod_{i=1}^k \prod_{j=1}^n \chi_{j}(B_i)^{\left \lfloor \frac {mj}n \right \rfloor - \left \lfloor \frac {m(j-1)}n \right \rfloor}$, which we will do by descending induction over $\lambda$. By definition,
\begin{equation}
\label{eqn:george}
d^\lambda = \sum_{\nu} s_\nu^{m/n}|_\lambda \cdot b^\nu = \sum_\nu c_\nu^\lambda(q,t) \cdot b^\nu
\end{equation}
By the definition of the stable basis, the only terms which contribute to the above sum are $\nu \unrhd \lambda$. Moreover, all the weights $q^\alpha t^\beta$ which appear in the RHS for $\nu \rhd \lambda$ have the property that $\alpha - \beta$ lies in the interval:
$$
\left[ \frac mn(o_\lambda - o_\nu) + {\min}_\lambda, \ \frac mn(o_\lambda - o_\nu) + {\max}_\lambda \right) + \text{weight of }b^\nu
$$
which by the induction hypothesis equals the interval: 
\begin{equation}
\label{eqn:interval}
\left[ \frac mn(o_\lambda - o_\mu) + {\min}_\lambda + \frac {k(n-1)}2, \ \frac mn(o_\lambda - o_\mu) + {\max}_\lambda + \frac {k(n-1)}2 \right)
\end{equation}
By the assumption in the claim, the weights of $d^\lambda$ lie in the same interval, but they are allowed to also equal the right endpoint of the interval. Therefore, the only term in \eqref{eqn:george} which can produce the right endpoint of the interval is when $\nu = \lambda$, in which case the definition of the stable basis tells us that the highest degree term of $c_\lambda^\lambda(q,t)$ equals $\prod_{\square \in \lambda} (-q^{a(\square)+1})$. We conclude that
$$
\hd d^\lambda = b^\lambda \cdot (-1)^{|\lambda|} q^{\max_\lambda} 
$$
which proves the induction step. Note that, in the above argument, we used the fact that $b^\lambda$ only consists of monomials $q^{\alpha}t^{\beta}$ for a single value of $\alpha-\beta$. This is because otherwise, the interval of exponents $\alpha-\beta$ which would appear in $c^\lambda_\lambda(q,t) \cdot b^\lambda$ would have length strictly greater than \eqref{eqn:interval}, thus contradicting \eqref{eqn:george}. 

\end{proof}

\text{} \\
We will finish by proving inequality \eqref{eqn:inequality0} and relation \eqref{eqn:lowestorder}, and leave \eqref{eqn:inequality} as an analogous exercise. By the definition of the stable basis in \eqref{eqn:ord1}, we have:
\begin{equation}
\label{eqn:square}
\degu c_\mu^\nu(q,t) \leq \frac mn(o_\nu - o_\mu) + {\max}_\nu,
\end{equation}
with equality if and only if $\mu = \nu$, while Proposition \ref{prop:main} gives us:
$$
\degu E_k^{m/n}(\lambda \backslash \nu) \leq \frac mn(o_\lambda - o_\nu) + \#_{\lambda \backslash \nu} + \frac {k(n-1)}2
$$
It is straightforward to compute the upper and lower limits of a product of linear factors, so note that we have: 
$$
\degu \prod_{\blacksquare \in \lambda \backslash \nu} \left[ \left(t - q \chi_\blacksquare \right) \prod_{\square \in \nu} \omega \left(\frac {\chi_{\blacksquare}}{\chi_{\square}} \right) \right] = \sum_{\blacksquare \in \lambda \backslash \nu}\left[ 1 +  \max(0,o_\blacksquare) + \right.
$$
$$
\left. + \# \text{ of }\square \in \nu \text{ on the same diagonal as } \blacksquare \right] = {\max}_\lambda - {\max}_\nu - \#_{\lambda \backslash \nu}
$$
Adding the above three inequalities gives us precisely the desired \eqref{eqn:inequality0}. As for \eqref{eqn:lowestorder}, we need to compare the highest order terms. Since:
\begin{equation}
\label{eqn:hlemmur}
\hd d^\lambda = \hd  c_\mu^\mu(q,t) \cdot \hd E_k^{m/n}(\lamu) \cdot \hd \prod_{\blacksquare \in \lambda \backslash \mu} \left[ \left(t - q \chi_\blacksquare \right) \prod_{\square \in \mu} \omega \left(\frac {\chi_{\blacksquare}}{\chi_{\square}} \right) \right] \qquad \quad
\end{equation}
we need to compute the three highest degree terms in the RHS. \footnote{The reason why only $c_\mu^\mu(q,t)$ contributes to the highest degree term follows from the fact that equality is attained in \eqref{eqn:square} only for $\mu = \nu$} By \eqref{eqn:normal}, we have:
\begin{equation}
\label{eqn:sterling}
\hd c_\mu^\mu(q,t) = (-1)^{|\mu|} q^{\max_\mu}
\end{equation}
It is straightforward to compute the lowest order term of a product of linear factors:
$$
\hd \prod_{\blacksquare \in \lambda \backslash \mu} \left[ \left(t - q \chi_\blacksquare \right) \prod_{\square \in \mu} \omega \left(\frac {\chi_{\blacksquare}}{\chi_{\square}} \right) \right] = (-t)^{kn}  \hd \prod_{\blacksquare \in \lambda \backslash \mu} \frac {\prod^{\square\text{ inner}}_{\text{corner of }\mu} \left(\frac {q \chi_\blacksquare}{t\chi_\square} - 1 \right)}{\prod^{\square\text{ outer}}_{\text{corner of }\mu} \left(\frac {q \chi_\blacksquare}{t\chi_\square} - 1\right)} = 
$$
$$
= (-t)^{kn} \frac {\prod^{\square\text{ inner}}_{\text{corner of }\mu} \left[\prod^{\blacksquare \in \lamu}_{o_\blacksquare > o_\square} \left( \frac {q \chi_\blacksquare}{t \chi_\square} \right) \prod^{\blacksquare \in \lamu}_{o_\blacksquare = o_\square} \left(\frac {q \chi_\blacksquare}{t \chi_\square} - 1\right) \prod^{\blacksquare \in \lamu}_{o_\blacksquare < o_\square} (-1) \right]}{\prod^{\square\text{ outer}}_{\text{corner of }\mu} \left[\prod^{\blacksquare \in \lamu}_{o_\blacksquare > o_\square} \left(\frac {q \chi_\blacksquare}{t\chi_\square} \right) \prod^{\blacksquare \in \lamu}_{o_\blacksquare = o_\square} \left(\frac {q \chi_\blacksquare}{t\chi_\square} - 1 \right) \prod^{\blacksquare \in \lamu}_{o_\blacksquare < o_\square} (-1) \right]} = (-t)^{kn} \cdot
$$
$$
\frac {\prod^{\square\text{ inner}}_{\text{corner of }\mu} \prod^{\blacksquare \in \lamu}_{o_\blacksquare = o_\square} \left(\frac {q \chi_\blacksquare}{t\chi_\square} - 1 \right)}{\prod^{\square\text{ outer}}_{\text{corner of }\mu} \prod^{\blacksquare \in \lamu}_{o_\blacksquare = o_\square} \left(\frac {q \chi_\blacksquare}{t\chi_\square} - 1\right)} \prod_{\blacksquare \in D_1 \cup D_3} q^{x_0} \prod_{\blacksquare \in D_2 \cup D_4} q^{x+1} t^{y_0-y-1} \prod_{\blacksquare \in D_2 \cup D_3} (-1)
$$
where the sets $D_1$, $D_2$, $D_3$, or $D_4$ consist of those boxes $\blacksquare = (x,y) \in \lamu$ that are on the same diagonal as a box $\square = (x_0,y_0)$ which is an inner corner, outer corner, vertical inside edge or horizontal inside edge of $\lamu$. We force a factor out of the above formula to obtain:
$$
= (-t)^{kn} \left(\frac tq \right)^{\flat_\lamu} \frac {\prod^{\square\text{ inner}}_{\text{corner of }\mu} \prod^{\blacksquare \in \lamu}_{o_\blacksquare = o_\square} \left(\frac {q \chi_\blacksquare}{t\chi_\square} - 1 \right)}{\prod^{\square\text{ outer}}_{\text{corner of }\mu} \prod^{\blacksquare \in \lamu}_{o_\blacksquare = o_\square} \left(\frac {q \chi_\blacksquare}{t\chi_\square} - 1\right)} (-1)^{|D_2|+|D_3|} \prod_{\blacksquare \in \lamu} q^{x+1}t^{y_0-y-1}   
$$
where recall that $\flat_\lamu$ counts the number of $\blacksquare$ which are $p$ steps in the northeast direction from an inner corner or vertical inside edge, counted with multiplicity $p+1$. Here we have used the simple observation that:
$$
\sum_{\blacksquare = (x,y) \in D_1 \cup D_3} (x-x_0+1) = \sum_{\blacksquare = (x,y) \in D_1 \cup D_3} (y-y_0+1) = \flat_\lamu
$$
But now observe that:
$$
\sum_{\blacksquare = (x,y) \in \lamu} (x+1) = {\max}_\lambda -{\max}_\mu, \qquad \sum_{\blacksquare = (x,y) \in \lamu} (y_0-y-1)  = - \#_\lamu - kn  
$$
so when $\lamu$ is a vertical $k-$strip of $n-$ribbons $B_1,...,B_k$, the highest degree term in the above formula amounts to:
$$
= (-1)^{kn}  \left(\frac tq \right)^{\flat_\lamu} \frac {\prod^{\square\text{ inner}}_{\text{corner of }\mu} \prod^{\blacksquare \in \lamu}_{o_\blacksquare = o_\square} \left(\frac {q \chi_\blacksquare}{t\chi_\square} - 1\right)}{\prod^{\square\text{ outer}}_{\text{corner of }\mu} \prod^{\blacksquare \in \lamu}_{o_\blacksquare = o_\square} \left(\frac {q \chi_\blacksquare}{t\chi_\square} - 1\right)} (-1)^{\high} q^{\max_\lambda-\max_\mu} t^{-\#_\lamu} 
$$
Multiplying the above factor with \eqref{eqn:lana} and \eqref{eqn:sterling} gives us the desired \eqref{eqn:hlemmur}.

\end{proof}

\section{Connection to LLT polynomials}
\label{sec:action}

\subsection{} For an $n-$ribbon $B$, let us denote by:
$$
\theta_m(B) = (-1)^{\high B}  \prod_{j=1}^n \chi_{j}(B)^{\left \lfloor \frac {mj}n \right \rfloor - \left \lfloor \frac {m(j-1)}n \right \rfloor}
$$
the coefficients that appear in \eqref{eqn:pieri}. We will now recall the framework of LLT polynomials, developed in \cite{LLT} in the case $m=0$. The generalization to arbitrary $m$ is straightforward, and we will show it to be equivalent to the original picture. If we iterate \eqref{eqn:pieri}, we obtain for any composition $\nu = (\nu_1,...,\nu_t)$:
\begin{equation}
\label{eqn:comp}
e_{\nu_1}^{m/n} \ldots e_{\nu_t}^{m/n}\cdot s^{m/n}_\mu = \sum_\lambda s^{m/n}_\lambda \sum_{(B_1,\ldots, B_{|\nu|}) \in K_\lamu^\nu} \prod_{i=1}^{|\nu|} \theta_m(B_i)
\end{equation}
where $K_\lamu^\nu$ is the set of \textbf{ribbon tableaux} of shape $\lamu$ and weight $\nu$. By definition, such a ribbon tableau consists of a vertical $\nu_1-$strip of $n-$ribbons on top of a vertical $\nu_2-$strip of $n-$ribbons $\ldots$ on top of a vertical $\nu_t-$strip of $n-$ribbons. We write:
$$
\high T = \sum_{i=1}^{|\nu|} \high B_i, \qquad \qquad \theta_m(T) = \prod_{i=1}^{|\nu|} \theta_m(B_i)
$$
Then let us construct the generating series of such ribbon tableaux:
\begin{equation}
\label{eqn:ribbon}
G_\lamu^{m/n}(x_1,x_2,\ldots) = \sum_{\nu} x_1^{\nu_1} \ldots x_t^{\nu_t}  \sum_{T \in K_\lamu^\nu} \theta_m(T)
\end{equation}

\begin{lemma}
The function $G_\lamu^{m/n}$ is \textbf{symmetric} in the variables $x_1,x_2,\ldots$. \\
\end{lemma}

\begin{proof} Let us write $A_\mu^\lambda$ for the coefficient of $s_\lambda^{m/n}$ in the expansion of $A \cdot s_\mu^{m/n}$, for any operator $A$. Then we have:
$$
G_\lamu^{m/n}(x_1,\ldots, x_t) = \sum_{\nu} x_1^{\nu_1} \ldots x_t^{\nu_t} \left(e_{\nu_1}^{m/n} \ldots e_{\nu_t}^{m/n}\right)_\mu^{\lambda}
$$ 
Because the operators $e_k^{m/n}$ all commute, we may change their order in the above product for any permutation $\sigma \in S(t)$, and so:
$$
G_\lamu^{m/n}(x_1,\ldots, x_t) = \sum_{\nu} x_1^{\nu_1} \ldots x_t^{\nu_t} \left(e_{\nu_{\sigma(1)}}^{m/n} \ldots e_{\nu_{\sigma(t)}}^{m/n}\right)_\mu^{\lambda} = 
$$ 
$$
= \sum_{\nu} x_{\sigma(1)}^{\nu_1} \ldots x_{\sigma(t)}^{\nu_t} \left(e_{\nu_1}^{m/n} \ldots e_{\nu_t}^{m/n}\right)_\mu^{\lambda} = G_\lamu^{m/n}(x_{\sigma(1)},\ldots, x_{\sigma(t)})
$$

\end{proof}

\subsection{} The original LLT polynomials were defined in \cite{LLT} as:
\begin{equation}
\label{eqn:llt}
\tG_\lamu(x_1,x_2,\ldots) = \sum_{\nu} x_1^{\nu_1} \ldots x_t^{\nu_t}  \sum_{T \in K_\lamu^\nu}  (-s)^{\high T}
\end{equation}
where $s$ is a formal parameter. Although the above coefficient looks different from \eqref{eqn:ribbon}, we will show that the two generating series are actually the same, up to constant multiple. \\

\begin{proposition}
\label{prop:ribbons}

Let $s = \sqrt{\frac tq}$. For any ribbon tableau $T$ of shape $\lamu$, we have: 
\begin{equation}
\label{eqn:ribbons}
\frac {\theta_m(T)}{(-s)^{\emph{ht } T}} = \gamma_{\lamu}
\end{equation}
where the constant $\gamma_\lamu$ only depends on the skew diagram, and not on the tableau.

\end{proposition}

\text{} \\
We will prove the above Proposition in the next Subsection, where we will also describe the constant $\gamma_\lamu$. Let us note that the above immediately implies that:
\begin{equation}
\label{eqn:equality}
G_\lamu^{m/n}(x_1,x_2,\ldots) = \gamma_\lamu \tG_\lamu (x_1,x_2,\ldots) \Big |_{s\rightarrow \sqrt{\frac tq}}
\end{equation}

\subsection{} Let us fix $\lamu$, and consider the graph whose vertices are all ribbon tableaux. We draw an oriented edge from one vertex to another if the former can be obtained from the latter by a \textbf{collapse} of one constituent ribbon into another:

\begin{picture}(300,150)(-5,10)
\label{collapse}

\put(142,95){\text{collapse}}
\put(135,90){\vector(1,0){50}}

\put(30,145){\line(0,-1){15}}\put(45,145){\line(0,-1){15}}\put(60,145){\line(0,-1){15}}
\put(60,130){\line(1,0){15}}\put(60,115){\line(1,0){15}}\put(75,115){\line(0,-1){15}}
\put(90,115){\line(0,-1){15}}\put(90,100){\line(1,0){15}}\put(90,85){\line(1,0){15}}
\put(90,70){\line(1,0){15}}\put(105,70){\line(0,-1){15}}\put(120,70){\line(0,-1){15}}

\put(45,115){\line(1,0){15}}\put(45,100){\line(1,0){15}}\put(60,100){\line(0,-1){15}}
\put(75,100){\line(0,-1){15}}\put(75,85){\line(1,0){15}}\put(75,70){\line(1,0){15}}
\put(75,55){\line(1,0){15}}\put(90,55){\line(0,-1){15}}\put(105,55){\line(0,-1){15}}
\put(120,55){\line(0,-1){15}}\put(135,55){\line(0,-1){15}}\put(135,40){\line(1,0){15}}

\put(200,145){\line(0,-1){15}}\put(215,145){\line(0,-1){15}}\put(230,130){\line(-1,0){15}}
\put(230,130){\line(1,0){15}}\put(230,115){\line(1,0){15}}\put(245,115){\line(0,-1){15}}
\put(260,115){\line(0,-1){15}}\put(260,100){\line(1,0){15}}\put(260,85){\line(1,0){15}}
\put(260,70){\line(1,0){15}}\put(275,70){\line(0,-1){15}}\put(290,70){\line(0,-1){15}}

\put(215,115){\line(1,0){15}}\put(215,100){\line(1,0){15}}\put(230,100){\line(0,-1){15}}
\put(245,100){\line(0,-1){15}}\put(245,85){\line(1,0){15}}\put(245,70){\line(1,0){15}}
\put(245,55){\line(1,0){15}}\put(260,55){\line(0,-1){15}}\put(275,55){\line(0,-1){15}}
\put(290,55){\line(1,0){15}}\put(305,55){\line(0,-1){15}}\put(305,40){\line(1,0){15}}

\linethickness{0.8mm}

\put(15,145){\line(0,-1){15}}\put(15,145){\line(1,0){60}}\put(75,145){\line(0,-1){30}}
\put(75,115){\line(1,0){30}}\put(105,115){\line(0,-1){45}}\put(105,70){\line(1,0){30}}
\put(135,70){\line(0,-1){15}}\put(135,55){\line(-1,0){45}}\put(90,55){\line(0,1){45}}
\put(90,100){\line(-1,0){30}}\put(60,100){\line(0,1){30}}\put(60,130){\line(-1,0){45}}

\put(45,130){\line(0,-1){45}}\put(45,85){\line(1,0){30}}\put(75,85){\line(0,-1){45}}
\put(75,40){\line(1,0){60}}\put(135,40){\line(0,-1){15}}\put(135,25){\line(1,0){15}}
\put(150,25){\line(0,1){30}}\put(150,55){\line(-1,0){15}}

\put(185,145){\line(0,-1){15}}\put(185,145){\line(1,0){60}}\put(245,145){\line(0,-1){30}}
\put(245,115){\line(1,0){30}}\put(275,115){\line(0,-1){45}}\put(275,70){\line(1,0){30}}
\put(305,70){\line(0,-1){15}}\put(290,55){\line(-1,0){30}}\put(260,55){\line(0,1){45}}
\put(260,100){\line(-1,0){30}}\put(230,100){\line(0,1){45}}\put(215,130){\line(-1,0){30}}

\put(215,130){\line(0,-1){45}}\put(215,85){\line(1,0){30}}\put(245,85){\line(0,-1){45}}
\put(245,40){\line(1,0){60}}\put(305,40){\line(0,-1){15}}\put(305,25){\line(1,0){15}}
\put(320,25){\line(0,1){30}}\put(320,55){\line(-1,0){15}}\put(290,55){\line(0,-1){15}}

\put(0,135){$B_1$}
\put(30,120){$B_2$}
\put(50,135){$i$}

\put(172,135){$B$}
\put(248,135){$B'$}
\put(233,135){$X$}

\put(150,5){\text{Figure 4}}

\end{picture}


\text{} 

\begin{proof} \textbf{of Proposition \ref{prop:ribbons}:} The above will be called the ``collapse graph" of shape $\lamu$. In order to prove the Proposition, we need to prove two statements: that a collapse $T\rightarrow T'$ does not change the quantity in the LHS of \eqref{eqn:ribbons}, and that the collapse graph is connected. For the first statement, it is easy to see that:
$$
\high T' - \high T = 2
$$
Let us denote by $B_1$ the ribbon above $B_2$ that are being collapsed in $T$, and assume that their first common edge is adjacent to the $i-$th box of $B_1$, as in Figure 4. Then we have:
$$
\frac {\theta_m(T')}{\theta_m(T)} = \frac {(-1)^{\high T'}(\chi_1qt^{-1})^{r(1)}...(\chi_{n-i}qt^{-1})^{r(n-i)}\chi_{1}^{r(i+1)}... \chi_{n-i}^{r(n)}}{(-1)^{\high T} \chi_1^{r(1)}...\chi_{n-i}^{r(n-i)}(\chi_{1}qt^{-1})^{r(i+1)}... (\chi_{n-i}qt^{-1})^{r(n)}} = \frac tq
$$
where $\chi_1,...,\chi_n$ denote the contents of the boxes of $B_2$ before the collapse, and:
$$
r(j) = \left \lfloor \frac {mj}n \right \rfloor - \left \lfloor \frac {m(j-1)}n \right \rfloor \Longrightarrow r(i+1)+...+r(n)-r(1)-..-r(n-i) = 1 
$$
for all $i\in \{1,...,n-1\}$. Since $s^2 = \frac tq$, this implies that the LHS of \eqref{eqn:ribbons} does not change upon a collapse. We still need to prove that the collapse graph is connected, which follows from the claim: \\
 
\begin{claim}
 	
Any $n-$ribbon tableau of any shape $\lamu$ is connected in the collapse graph to the \textbf{minimal tableau}.
 	
\end{claim}

\text{} \\
We define the minimal tableau by starting at the northwestern most box $\square$ of the skew diagram, and tracing the \textbf{outer chain} of $n-$ribbons along the outer boundary of the partition. When we arrive at the southeastern most box $\blacksquare$ with the outer chain, we remove the chain and obtain a skew diagram of smaller size, to which we repeat the construction. Note that it is not clear that the minimal tableau is well-defined (a priori, the outer chain may not end precisely at $\blacksquare$), but this will be proved together with the claim by induction on the size $\lamu$ of a skew diagram which admits any $n-$ribbon covering.
 
\text{} \\
We will prove that any $n-$ribbon tableau $T$ is connected to one where the ribbon $B$ containing the northwestern most box $\square$ lies completely along the outer boundary of $\lamu$. Then it will be clear that the argument can be continued to form the outer chain, and the claim will follow by induction. Let us construct a sub-tableau $T_0 \subset T$ as the minimal collection of ribbons that includes $B$ and such that if a ribbon $B_0\in T$ touches an outer boundary \footnote{The outer boundary of a ribbon consists of those boundary edges above and to the right of it} of a ribbon in $T_0$, then $B_0 \in T_0$. It is easy to see that $T_0 \backslash B$ is in the shape of a skew diagram $\nu \subset \lamu$, so by the induction hypothesis one can transform it (via collapses and their inverses) into the minimal tableau of shape $\nu$. By the very definition of the minimal tableau, its northwestern most ribbon $B_0$ can be successively collapsed until it touches the inner boundary of $\nu$. But then $B$ and $B_0$ are precisely in the position on the right of Figure 4, so an inverse collapse allows $B$ to include the box marked $X$ therein. Repeating this procedure, we will ensure that $B$ will lie completely along the outer boundary of $\lamu$. 








\end{proof}

\subsection{} We will still write $s = \sqrt {t/q}$. The algebra $\CA$ of Subsection \ref{sub:elliptic} is the positive half of its Drinfeld double:
$$
D\CA = \langle p_{v} \rangle_{v \in \BZ^2 \backslash (0,0)}
$$
The algebra $D\CA$ has twice as many generators as $\CA$, and relations \eqref{eqn:rel1} and \eqref{eqn:rel2} need to be slightly amended:
\begin{equation}
\label{eqn:heisenberg}
[p_{kn,km}, p_{ln,lm}] = \frac {k\delta_{k+l}^0 \cdot (1-q^{-k})}{(q^k-t^k)(1-t^{-k})} \cdot (s^{-kn} - s^{kn})
\end{equation}
for any $k,l>0$ and any coprime $m,n$, and: 
$$
[p_{v}, p_{v'}] = s^{\alpha(v,v')} \theta_{v+v'}
$$
in the notation of Subsection \ref{sub:elliptic}. In the above, $\alpha(v,v')$ denotes a certain integer that needs to be added to the formula when the lattice points $v$ and $v'$ lie on opposite sides of the vertical axis. We will not review the exact definition here, as we will not need it, but the interested reader may find it in \cite{Nsurv}. It is also shown in \loccit \ that the whole double algebra $D\CA$ acts on $\Lambda$, where:
\begin{equation}
\label{eqn:adjoint}
p_{-n,m} = - p_{n,m}^\dagger \cdot \left( \frac sq \right)^n
\end{equation}
for all $n>0$. The adjoint is taken with respect to the inner product \eqref{eqn:mac}. \\


\subsection{} By analogy with our notation $p_k^{m/n} = p_{kn,km}$, we will write:
$$ 
p_{-k}^{m/n} = p_{-kn,-km} \cdot \left(-s^{k} \frac {1-t^k}{1-q^{-k}}\right)
$$
By \eqref{eqn:heisenberg}, we have a quantum Heisenberg type relation:
$$
\left[ p^{m/n}_k, p^{m/n}_l \right] =  k\delta_{k+l}^0 \cdot \frac {s^{kn} - s^{-kn}}{s^k-s^{-k}}, \qquad \quad \forall k,l\in \BZ
$$
If we think of $p_{-k}^{m/n}$ as power sum functions, we define the operators $e_{-k}^{m/n}$ to be the corresponding elementary symmetric functions. The analogue of Theorem \ref{thm:stab} is: \\

\begin{theorem}
\label{thm:stab2}

For any $(m,n) \in \BZ \times \BN$ and any positive integer $k$, we have:
\begin{equation}
\label{eqn:pieri2}
e_{-k}^{m/n} \cdot s^{m/n}_{\lambda} = \frac {(-1)^{kn}}{(qt)^{\frac {k(n-1)}2}} \sum s^{m/n}_\mu  (-1)^{\emph{width}} \prod_{i=1}^k\prod_{j=1}^n \chi_{j}(R_i)^{\left \lfloor \frac {-mj}n \right \rfloor - \left \lfloor \frac {-m(j-1)}n \right \rfloor} 
\end{equation}
where the sum goes over all horizontal $k-$strips of $n-$ribbons of shape $\lamu$, and the remaining notations are as in Theorem \ref{thm:stab}.

\end{theorem}

\textbf{} \\
A horizontal strip of ribbons is defined just like a vertical strip of ribbons in Subsection \ref{sub:ribbon}. The proof of the above Theorem goes through almost word by word as that of Theorem \ref{thm:stab}, and so we will not repeat it. The starting point is the same formula \eqref{eqn:shuf}, which tells us how the operators $e_k^{m/n}$ act in the basis of renormalized Macdonald polynomials. By \eqref{eqn:adjoint}, we have:
$$
e_{-k}^{m/n} \sim \left( e_k^{-m/n} \right)^\dagger
$$
up to an overall constant, and so the matrix coefficients of the LHS in the basis of Macdonald polynomials are obtained from \eqref{eqn:shuf} by switching $M_\mu$ and $M_\lambda$ (up to a product of linear factors and an overall constant). This accounts for two differences between the present formula \eqref{eqn:pieri2} and the $m/n$ Pieri rule \eqref{eqn:pieri}: \\

\begin{itemize} 

\item we have $-m$ in the exponents of Theorem \ref{thm:stab2}, as opposed from the exponent $m$ that appeared in Theorem \ref{thm:stab} \\

\item we consider horizontal strips of ribbons (and their width) instead of vertical strips of ribbons (and their height). The reason for this is the $-$ sign in \eqref{eqn:adjoint}, which tells us that we need to flip the sign of power sum functions as we go from positive to negative. This has the effect of replacing elementary symmetric by complete symmetric functions. \\

\end{itemize}

\section{Appendix}
\label{sec:app}

\begin{proof} \textbf{of Proposition \ref{prop:well}:} As we read from top to bottom, the rows of the Young diagram $\lamu$ are horizontal strips $D_1,...,D_l$ of boxes, where we write:
	$$
	D_i = \{(x_i,y_i),(x_i+1,y_i),...,(x_i'-1,y_i),(x_i',y_i)\}	
	$$
	Thus, let us consider the following rational functions in $l$ variables:
	$$
	R'(a_1,...,a_l) = R(...,a_i q^{x_i},...,a_i q^{x_i'},....) = \frac {r(a_1 q^{x_1},...,a_1 q^{x_1'},....,a_l q^{x_l},...,a_l q^{x_{l'}})}{\prod^{(i,j)\neq (i',j'), 1\leq i,i'\leq l}_{x_i \leq j \leq x_i', \ x_{i'} \leq j' \leq x'_{i'}} (a_i t q^j- a_{i'} q^{j'+1})}
	$$
for certain indeterminates $a_1,...,a_l$, where we use formula \eqref{eqn:shuf} in the second equality. Since the Laurent polynomial $r$ satisfies the wheel conditions, the numerator vanishes when $a_i t q^j = a_{i'}q^{j'+1}$ for $x_{i'} \leq j' < x_{i'}'$ and when $a_i q^{j+1} = a_{i'} t q^{j'}$ for $x_{i'} < j' \leq x_{i'}'$, for all $1\leq i<i' \leq l$. This means that the numerator will be divisible by linear factors involving the $a_i$, which means the above formula can be simplified to:
	$$
	R'(a_1,...,a_l) = \frac {r'(a_1,...,a_l)}{\prod^{1\leq i<i' \leq l}_{x_i\leq j \leq x_i'} (a_i t q^j- a_{i'} q^{x'_{i'}+1}) (a_i q^{j+1}- a_{i'} t q^{x_{i'}})}
	$$
	for some Laurent polynomial $r'$. To obtain the value $R(\lamu)$, we need to set $a_i = t^{-y_i}$ in the above, and note that the ratio is well-defined as long as $t^{-y_i+1}q^j \neq t^{-y_{i'}}q^{x'_{i'}+1}$ and $t^{-y_i}q^{j+1} \neq t^{-y_{i'}+1}q^{x_{i'}}$ for any $i<i'$ and $x_i \leq j \leq x_i'$. This is equivalent to the following condition, obviously satisfied for a skew Young diagram:
	$$
	\textbf{Condition:} \text{ the box directly to the right/left of the strip }D_{i'}
	$$
	$$
	\text{ does not lie directly below/above the strip }D_i, \forall \ i<i'
	$$
	
\end{proof}

\begin{proof} \textbf{of Proposition \ref{prop:eval}:} Let us write $R = P_k^{m/n}$ and $R' = P_k^{-m/n}$ throughout the proof, and we will prove by induction on $l$ that:
	$$
	R(\mu_l) = c_l\sum_{i=0}^{n-1} q^{\sum_{j=1}^{nk-l} \left \lfloor \frac {mj+i}n \right \rfloor} t^{-\sum_{j=1}^{l-1} \left \lceil \frac {mj-i}n \right \rceil}
	$$
	where $c_l = \frac {1-q^{-k}}{1-tq^{-1}} \prod_{i=1}^{nk-l} \frac {1-q^{-i}}{1-q^{-i-1}t} \prod_{i=1}^{l-1} \frac {1-t^i}{1-t^{i+1}q^{-1}}$. The base case $l=1$ of the induction is simply the fourth bullet of Subsection \ref{sub:bull}. We set $\rho = (1,...,1,nk-l+1)$ and note that:
	\begin{equation}
	\label{eqn:deneuve}
	R(\mu_l) = R^\rho(y_1,...,y_l) \ \Big |_{y_i = t^{i-l}}	
	\end{equation}
	in the notation from the proof of Lemma \ref{lem:uniqe}. We have:
	$$
	\Delta(p_k) = p_k \otimes 1 + 1 \otimes p_k \quad \Longrightarrow \quad \Delta_{m/n}(R) = R \otimes 1 + 1 \otimes R	
	$$
	which implies that $R_{1,i} = R_{2,i} = 0$ for all $i \in (0,k)$. Hence when we use \eqref{eqn:lagrange} to evaluate the RHS of \eqref{eqn:deneuve}, only the second term appears:
	\begin{equation}
	\label{eqn:lagy}
	R^\rho(y_1,...,y_l) = \sum^{i<l}_{a \in S^\pm_{il}} \gamma^{\pm}_{i,a} \cdot R^{\rho(i \stackrel{a}\leftrightarrow l)}(y_1,...,y_l)|_{y_l q^a = y_i}  
	\end{equation}
	in the notation of \eqref{eqn:lagrange}. In the case at hand, we have $\rho = (1,...,1,nk-l+1)$, and hence it is easy to see that the sets $S_{il}^\pm$ are given by:
	$$
	S_{il}^+ = \{nk-l+1\}, \qquad S_{il}^- = \{-1\}, \qquad \forall i \in \{1,...,l-1\}
	$$
	By \eqref{eqn:deneuve}, we need to evaluate the RHS \eqref{eqn:lagy} at $y_i = t^{i-l}$. We observe that all the constants $\gamma_{i,a}^\pm$ (where $a$ is the unique element of $S_{il}^\pm$) vanish, except for:
	$$
	\gamma_{1l}^+ \ \Big |_{y_i = t^{i-l}} = \frac {c_l}{c_{l-1}} \cdot \left(t^{l-1}q^{nk-l+1}\right)^{\left \lfloor \frac {m(nk-l+1)}n \right \rfloor} \frac {1-tq^{-1}}{1-t^{2-l}q^{-nk+l-2}} 
	$$
	and:
	$$
	\gamma_{l-1,l}^- \ \Big |_{y_i = t^{l-i}} = \frac {c_l}{c_{l-1}} \cdot \left(tq^{-1} \right)^{\left \lfloor \frac {m(nk-l+1)}n \right \rfloor + 1} \frac {1 - t^{1-l} q^{-nk+l-1}}{1 - t^{2-l} q^{-nk+l-2}} 
	$$
	With this in mind, \eqref{eqn:lagrange} and \eqref{eqn:lagy} imply that:
	$$
	R(\mu_l) = \gamma_{1l}^+ \ \Big |_{y_i = t^{i-l}} \cdot R(q^{-nk+l-1}t^{1-l},...,q^{-1}t^{1-l}, t^{1-l},t^{2-l},...,t^{-1}) + 
	$$
	$$
	+\gamma_{l-1,l}^- \ \Big |_{y_i = t^{l-i}} \cdot R(t^{1-l},...,t^{-1},t^{-1}q,...,t^{-1}q^{nk-l+1})
	$$
	We may apply homogeneity of $R$ and \eqref{eqn:inverse} to write the above as:
	$$
	R(\mu_l) = \gamma_{1l}^+ \ \Big |_{y_i = t^{i-l}} t^{(1-l)km} \cdot R'(\mu_{l-1})+ \gamma_{l-1,l}^- \ \Big |_{y_i = t^{l-i}} t^{-km} \cdot R(\mu_{l-1})
	$$
	Plugging in the values of $R(\mu_{l-1})$ and $R'(\mu_{l-1})$ from the induction hypothesis in the above gives us:
	$$
	R(\mu_l) = c_l \left[ \frac {\left(t^{l-1}q^{nk-l+1}\right)^{\left \lfloor \frac {m(nk-l+1)}n \right \rfloor} (1-tq^{-1})}{t^{(l-1)km}(1-t^{2-l}q^{-nk+l-2})} \sum_{i=0}^{n-1} q^{\sum_{j=1}^{nk-l+1} \left \lfloor \frac {-mj+i}n \right \rfloor} t^{-\sum_{j=1}^{l-2} \left \lceil \frac {-mj-i}n \right \rceil}  + \right.
	$$
	$$
	\left. + \frac {\left(tq^{-1} \right)^{\left \lfloor \frac {m(nk-l+1)}n \right \rfloor + 1}(1 - t^{1-l} q^{-nk+l-1})}{t^{km}(1-t^{2-l}q^{-nk+l-2})} \sum_{i=0}^{n-1} q^{\sum_{j=1}^{nk-l+1} \left \lfloor \frac {mj+i}n \right \rfloor} t^{-\sum_{j=1}^{l-2} \left \lceil \frac {mj-i}n \right \rceil} \right]
	$$
We wish to show that the RHS equals $c_l \sum_{i=0}^{n-1} q^{\sum_{j=1}^{nk-l} \left \lfloor \frac {mj+i}n \right \rfloor} t^{-\sum_{j=1}^{l-1} \left \lceil \frac {mj-i}n \right \rceil}$, i.e.:
$$
t^{(l-1)\left \lfloor \frac {m(1-l)}n \right \rceil} q^{(nk-l+1)\left \lfloor \frac {m(nk-l+1)}n \right \rfloor} (1-tq^{-1}) \sum_{i=0}^{n-1} q^{\sum_{j=1}^{nk-l+1} \left \lfloor \frac {-mj+i}n \right \rfloor} t^{\sum_{j=1}^{l-2} \left \lfloor \frac {mj+i}n \right \rfloor}  +
$$
$$
 + t^{\left \lfloor \frac {m(1-l)}n \right \rfloor + 1} q^{- \left \lfloor \frac {m(nk-l+1)}n \right \rfloor - 1} (1 - t^{1-l} q^{-nk+l-1}) \sum_{i=0}^{n-1} q^{\sum_{j=1}^{nk-l+1} \left \lfloor \frac {mj+i}n \right \rfloor} t^{\sum_{j=1}^{l-2} \left \lfloor \frac {-mj+i}n \right \rfloor} =
$$
\begin{equation}
\label{eqn:desire}
=(1-t^{2-l}q^{-nk+l-2}) \sum_{i=0}^{n-1} q^{\sum_{j=1}^{nk-l} \left \lfloor \frac {mj+i}n \right \rfloor} t^{\sum_{j=1}^{l-1} \left \lfloor \frac {-mj+i}n \right \rfloor}
\end{equation}
Proving \eqref{eqn:desire} will be an elementary computation. We will use the identities:
$$
\left \lfloor \frac {mz}n \right \rfloor + \left \lfloor \frac {-mj+i}n \right \rfloor= \left \lfloor \frac {m(z-j)}n \right \rfloor + \delta^{mj \text{ mod }n}_{\in \{1,...,i\}}  - \delta^{mj \text{ mod }n}_{\in \{1,...,mz \text{ mod }n\}} 
$$
$$
\left \lfloor \frac {mj+i}n \right \rfloor = \left \lfloor \frac {mj}n \right \rfloor + \delta^{-mj \text{ mod }n}_{\in \{1,...,i\}} 
$$
for all $g,i,j$, in order to obtain the following equalities for all $i<n$:
$$
q^{\sum_{j=1}^{nk-l+1} \left \lfloor \frac {m(nk-l+1)}n \right \rfloor + \left \lfloor \frac {-mj+i}n \right \rfloor} = q^{\sum_{j=1}^{nk-l} \left \lfloor \frac {mj}n \right \rfloor} \cdot q^{\sum_{j=1}^{nk-l+1} \delta^{mj \text{ mod }n}_{\in \{1,...,i\}} - \delta^{mj \text{ mod }n}_{\in \{1,...,m(1-l)\text{ mod }n\}}}
$$
$$
t^{\sum_{j=1}^{l-1} \left \lfloor \frac {m(1-l)}n \right \rfloor + \left \lfloor \frac {m(l-1-j)+i}n \right \rfloor} = t^{\sum_{j=1}^{l-1} \left \lfloor \frac {-mj}n \right \rfloor}  \cdot t^{\sum_{j=1}^{l-1} \delta^{-mj \text{ mod }n}_{\in \{1,...,i\}} - \delta^{-mj \text{ mod }n}_{\in \{1,...,m(1-l)\text{ mod }n\}}}
$$
$$
q^{\sum_{j=1}^{nk-l} \left \lfloor \frac {mj+i}n \right \rfloor} = q^{\sum_{j=1}^{nk-l} \left \lfloor \frac {mj}n \right \rfloor} \cdot q^{\sum_{j=1}^{nk-l} \delta^{-mj \text{ mod }n}_{\in \{1,...,i\}}}
$$
$$
t^{\sum_{j=1}^{l-1} \left \lfloor \frac {-mj+i}n \right \rfloor} = t^{\sum_{j=1}^{l-1} \left \lfloor \frac {-mj}n \right \rfloor} \cdot t^{\sum_{j=1}^{l-1}\delta^{mj \text{ mod }n}_{\in \{1,...,i\}}}
$$
With this, \eqref{eqn:desire} reduces to the following:
$$
(1-tq^{-1}) \sum_{i=0}^{n-1} \frac {q^{\sum_{j=1}^{nk-l+1} \delta^{mj \text{ mod }n}_{\in \{1,...,i\}}-\sum_{j=1}^{nk-l+1} \delta^{mj \text{ mod }n}_{\in \{1,...,m(1-l)\text{ mod }n\}}}}{t^{-\sum_{j=0}^{l-2} \delta^{-mj \text{ mod }n}_{\in \{1,...,i\}}+\sum_{j=0}^{l-2} \delta^{-mj \text{ mod }n}_{\in \{1,...,m(1-l)\text{ mod }n\}}}} =
$$
$$
= - (tq^{-1} - t^{2-l} q^{-nk+l-2}) \sum_{i=0}^{n-1} \frac {q^{\sum_{j=1}^{nk-l+1} \delta^{-mj \text{ mod }n}_{\in \{1,...,i\}} }}{t^{-\sum_{j=0}^{l-2}\delta^{mj \text{ mod }n}_{\in \{1,...,i\}}}} + (1-t^{2-l}q^{-nk+l-2}) \sum_{i=0}^{n-1} \frac {q^{\sum_{j=1}^{nk-l} \delta^{-mj \text{ mod }n}_{\in \{1,...,i\}} }}{t^{-\sum_{j=0}^{l-1}\delta^{mj \text{ mod }n}_{\in \{1,...,i\}}}}
$$
We may simplify the last line to reduce the above formula to:
$$
\sum_{i=0}^{n-1} \frac {q^{\sum_{j=0}^{nk-l+1} \delta^{mj \text{ mod }n}_{\in \{1,...,i\}}-\sum_{j=0}^{nk-l+1} \delta^{mj \text{ mod }n}_{\in \{1,...,x\}}}}{t^{-\sum_{j=1}^{l-2} \delta^{-mj \text{ mod }n}_{\in \{1,...,i\}}+\sum_{j=1}^{l-2} \delta^{-mj \text{ mod }n}_{\in \{1,...,x\}}}}=
$$
$$
= \sum_{i=0}^{n - x -1} \frac {q^{\sum_{j=0}^{nk-l+1} \delta^{-mj \text{ mod }n}_{\in \{1,...,i\}} }}{t^{-\sum_{j=1}^{l-2}\delta^{mj \text{ mod }n}_{\in \{1,...,i\}}}} + \sum_{i = n-x}^{n-1} \frac {q^{-\sum_{j=0}^{nk-l+1} \delta^{-mj \text{ mod }n}_{\in \{i+1,...,n\}} }}{t^{\sum_{j=1}^{l-2}\delta^{mj \text{ mod }n}_{\in \{i+1,...,n\}}}}
$$
where we write $x = m(1-l)$ mod $n$. The above equality is easily seen to hold term by term, with the summands corresponding to $i\geq x$ (respectively, $i\leq x$) in the LHS corresponding to the first (respectively, second) summand in the RHS.

\end{proof}

\end{document}